\documentclass[a4paper,11pt]{amsart}

\usepackage[latin1]{inputenc}
\usepackage[cyr]{aeguill}
\usepackage[english]{babel}
\usepackage{tikz,hyperref}
\usetikzlibrary{positioning}
\usepackage[T1]{fontenc}
\usepackage{amsmath,bbold}
\usepackage{amsfonts}
\usepackage{amssymb}
\usepackage{amsthm,eepic}
\usepackage{mathrsfs}
\usepackage{enumitem}
\usepackage{graphicx}
\usepackage{footnote}
\usepackage{breqn}
\usepackage{thm-restate}
\hypersetup{colorlinks=true,    linkcolor=blue,
citecolor=red,       filecolor=BrickRed,       urlcolor=darkgreen
}

\renewcommand{\geq}{\geqslant}
\renewcommand{\leq}{\leqslant}
\newtheorem{thm}{Theorem}

\newtheorem{prop}[thm]{Proposition}
\newtheorem{lem}{Lemma}

\newcommand{\be}{\begin{equation}}
\newcommand{\ee}{\end{equation}}

\colorlet{darkgreen}{green!50!black}
\definecolor{darkseagreen}{rgb}{0.56, 0.74, 0.56}
\definecolor{lightcyan}{rgb}{0.88, 1.0, 1.0}
\definecolor{lightblue}{rgb}{0.68, 0.85, 0.9}
\definecolor{palecerulean}{rgb}{0.61, 0.77, 0.89}
\definecolor{lgreen} {RGB}{180,210,100}
\definecolor{dblue}  {RGB}{20,66,129}
\definecolor{ddblue} {RGB}{11,36,69}
\definecolor{lred}   {RGB}{220,0,0}
\definecolor{nred}   {RGB}{224,0,0}
\definecolor{norange}{RGB}{230,120,20}
\definecolor{nyellow}{RGB}{255,221,0}
\definecolor{ngreen} {RGB}{98,158,31}
\definecolor{dgreen} {RGB}{78,138,21}
\definecolor{nblue}  {RGB}{28,130,185}
\definecolor{jblue}  {RGB}{20,50,100}
\definecolor{Apricot} {RGB}{255, 170, 123} 
\definecolor{dpurple}  {RGB}{53,21,93}

\def\C{{\mathbb C}}
\def\Z{{\mathbb Z}}
\def\1{\mathbb{1}}

\def\Z{\mathbb{Z}}
\def\C{\mathbb{C}}

\def\Q{\mathbb{Q}}
\renewcommand{\epsilon}{\varepsilon}

\renewcommand{\th}{\vartheta}
\newcommand{\thh}{\tilde{\vartheta}}

\usepackage{color}
\definecolor{darkgreen}{rgb}{0,0.4,0}
\definecolor{MyDarkBlue}{rgb}{0,0.08,0.50}
\definecolor{BrickRed}{rgb}{0.65,0.08,0}

\newcommand{\notea}[1]{\textcolor{red}{#1}}

\title[Logarithmic terms in discrete heat kernel expansions]{Logarithmic terms in discrete heat kernel expansions in~the~quadrant}

\author{Andrew Elvey Price}
\address{Institut Denis Poisson,
Universit\'e de Tours and Universit\'e d'Orl\'eans, CNRS, France}
\email{andrew.elvey@univ-tours.fr}

\author{Andreas Nessmann}
\address{Institut Denis Poisson,
Universit\'e de Tours and Universit\'e d'Orl\'eans, France and Institut f\"ur Diskrete Mathematik und Geometrie,
Technische Universit\"at Wien, Vienna, Austria}
\email{andreas.nessmann@tuwien.ac.at}

\author{Kilian Raschel}
\address{Laboratoire Angevin de Recherche en Math\'ematiques, Universit\'e d'Angers, CNRS, France}
\email{raschel@math.cnrs.fr}\thanks{This work has received funding from the European Research Council (ERC) under the European Union's Horizon
2020 research and innovation programme under the Grant Agreement No.~759702, from Centre Henri Lebesgue,
programme ANR-11-LABX-0020-0 and from the pr
oject DeRerumNatura, programme ANR-19-CE40-0018.}

\date{\today}

\begin{document}

\begin{abstract}
In the context of lattice walk enumeration in cones, we consider the number of walks in the quarter plane with fixed starting and ending points, prescribed step-set and given length. After renormalization, this number may be interpreted as a discrete heat kernel in the quadrant. We propose a new method to compute complete asymptotic expansions of these numbers of walks as their length tends to infinity, based on two main ingredients: explicit expressions for the underlying generating functions in terms of elliptic Jacobi theta functions along with a duality known as Jacobi transformation. This duality allows us to pass from a classical Taylor expansion of the series to an expansion at the critical point of the model. We work through two examples. First, we present our approach on the well-known Kreweras model, which is algebraic, and show how to obtain a complete asymptotic expansion in this case. We then consider a more generic (so-called infinite group) model, and find the associated complete asymptotic expansion. In this second case, we prove the existence of logarithmic terms in the asymptotic expansion, and we relate the coefficients appearing in the expansion to polyharmonic functions. To our knowledge, this is the first time that logarithmic terms have been observed in the asymptotics of a class of lattice walks confined to a quadrant. 
\end{abstract}

\maketitle

\section{Introduction and main results}

\noindent In the context of lattice walk enumeration in cones, we consider the number of walks $q(A,B;n)$ in a given cone with fixed starting and ending points, respectively denoted by $A$ and $B$, prescribed step-set (or transition probabilities) and given length $n$; the dependency of $q$ on the cone and the step-set is removed from our notation. The first-term asymptotics of $q(A,B;n)$ is known \cite{DeWa-15} for a wide class of cones and walk models, resulting in (ignoring periodicity)
\begin{equation}
\label{eq:asymptotic_DeWa-15}
   q(A,B;n) \sim \kappa v(A) \hat{v}(B) \frac{\mu^n}{n^\alpha},  
\end{equation}
where $\kappa>0$ is a universal constant, $\mu$ is the exponential growth of the model, $\alpha$ the critical exponent, and $v,\hat v$ certain discrete harmonic functions; see \cite[Eq.~(12)]{DeWa-15}. Preceding this general statement, a number of asymptotic results were obtained for specific cases, as will be recalled later on in the introduction (see Section~\ref{sec:earlier_asymptotics}).

The question of complete asymptotic expansions for lattice walk problems was addressed only recently. There are actually various relevant problems:
\begin{itemize}
    \item How will lower-order terms of the asymptotics of  $q(A,B;n)$ depend on the start and end points $A,B$?
    \item How can one access such complete asymptotic expansions?
    \item Will all appearing terms be a combination of exponentials and polynomials (such as $\frac{\mu^n}{n^\alpha}$), or will we observe the emergence of more complicated terms, such as logarithms?
\end{itemize}
These questions were the key motivations to the works \cite{ChFuRa-20,Ne-22,Ne-23}. More specifically, it is shown in \cite{ChFuRa-20,Ne-22} that lower-order terms should involve so-called discrete polyharmonic functions. Moreover, in \cite{Ne-23}, complete asymptotic expansions are obtained for a number of walk models, which all have the property that their reflection group is finite, and thus have an orbit-sum expression for the generating function which is suitable for saddle point analysis. In this context, only exponential-polynomial terms appear. So far the probabilistic method of \cite{DeWa-15} only provides one-term asymptotics, nonetheless, some progress on further terms is expected in dimension~$1$ \cite{DeTaWa-23}.

\subsection{A glimpse at our main results}
In this work, we propose a new method allowing us to address the three questions above. More precisely, we will analyse the two models of walks in the quarter plane represented in Figure~\ref{fig:step_sets}, which share the property that their generating function
\begin{equation}
\label{eq:def_series}
    Q(x,y) = Q(x,y;t):=\sum_{n\geq 0}\sum_{i,j\geq0} q((0,0),(i,j) ;n) x^iy^jt^n
\end{equation}
can be expressed in terms of the Jacobi theta function
\begin{align}
\label{eq:def_theta1}
    \th(z)=\th(z|\tau)&=\sum_{n=0}^{\infty}(-1)^n e^{i\pi\tau (n+\frac{1}{2})^{2}}\left(e^{(2n+1)iz}-e^{-(2n+1)iz}\right)\\=\th(z,q)&=2i\sum_{n=0}^{\infty}(-1)^n q^{(n+\frac{1}{2})^{2}}\sin\bigl((2n+1)z\bigr),\label{eq:def_theta2}
\end{align}
where $q=e^{i\pi\tau}$. Conveniently, the function \eqref{eq:def_theta2} is a series in $q$ which can be rapidly calculated. Alternatively, we can think of $\th(z|\tau)$ as an analytic function as long as $\tau\in\mathbb{C}$ has positive imaginary part, as this ensures that the series converges.

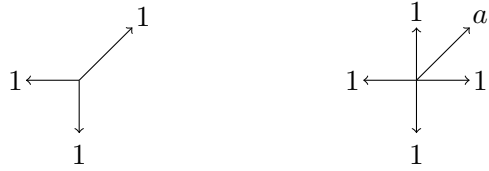
\begin{figure}
\begin{center}
\begin{tikzpicture}[scale=.7] 
    \draw[->,white] (1,2) -- (0,-2);
    \draw[->,white] (1,-2) -- (0,2);
    \draw[->] (0,0) -- (1,1);
    \draw[->] (0,0) -- (-1,0);
    \draw[->] (0,0) -- (0,-1);
    \node at (-1.2,0) {$1$};
    \node at (1.2,1.2) {$1$};
    \node at (0,-1.4) {$1$};
   \end{tikzpicture}\qquad\qquad\qquad
\begin{tikzpicture}[scale=.7] 
    \draw[->,white] (1,2) -- (0,-2);
    \draw[->,white] (1,-2) -- (0,2);
    \draw[->] (0,0) -- (1,1);
    \draw[->] (0,0) -- (1,0);
    \draw[->] (0,0) -- (0,-1);
    \draw[->] (0,0) -- (-1,0);
    \draw[->] (0,0) -- (0,1);
    \node at (-1.2,0) {$1$};
    \node at (1.2,1.2) {$a$};
    \node at (0,-1.4) {$1$};
    \node at (0,1.3) {$1$};
    \node at (1.2,0) {$1$};
   \end{tikzpicture}
\end{center}
\caption{The models considered in this work: on the left, the algebraic (unweighted)\ Kreweras model will serve as an example for our approach (Section~\ref{sec:Kreweras}); on the right, a more generic model with associated infinite reflection group (Section~\ref{sec:a-model}). In the second model, we allow a weight $a>0$ for the jump $(1,1)$, our motivation being to show how our results depend on this parameter, in particular in the limit $a\to0$, at which the model degenerates into the simple walk (to the four nearest neighbors).}
\label{fig:step_sets}
\end{figure}

More specifically, among all $79$ relevant quadrant walk models studied in \cite{BMMi-10}, $13$ have the property of admitting a so-called decoupling function, see \cite{BeBMRa-21}, which results in rather simple expressions for the generating functions \eqref{eq:def_series} using elliptic functions \cite{BeBMRa-21,EP-20,EP-22a,EP-22b}. Of these $13$ models, $4$ (resp.\ $9$) admit a finite (resp.\ an infinite) reflection group. While our approach would work for all these decoupled models, for brevity we choose to focus in the present work on the two examples of Figure~\ref{fig:step_sets}. See Propositions~\ref{prop:expression_Kreweras} and \ref{prop:expression_Q(x,0)_a-model} for instances of such  expressions in terms of theta functions.

To understand our main result, we mention here the Jacobi transformation, namely the symmetry $q\longleftrightarrow \hat q$ given by
\begin{equation}
\label{eq:Jacobi_transformation}
   \log(q)\log(\hat{q})=\pi^{2},
\end{equation}
and the associated the Jacobi identity
\begin{equation}
\label{eq:Jacobi_identity}
    \th(z,q)=\sqrt{-\frac{\log(\hat{q})}{\pi}}\exp\left(\frac{\log(\hat{q})z^{2}}{\pi^{2}}\right)\th\left(\frac{i}{\pi}\log(\hat{q})z,\hat{q}\right).
\end{equation}
One of our main contributions is that for the lattice walk models under consideration, this transformation admits a direct combinatorial interpretation, in that applying it on the theta-expression for the generating function, we will exactly obtain the expression of the function at the critical point. From here, using classical singularity analysis, we can deduce immediately a complete series asymptotic expansion. In other words:
\begin{itemize}
    \item $q=0$ ($\hat q=1$) $\longleftrightarrow$ series (Taylor) expansion of the series \eqref{eq:def_series} at $t=0$;
    \item $\hat q=0$ ($q=1$) $\longleftrightarrow$ series expansion of the series \eqref{eq:def_series} at the critical point $t=\frac{1}{\mu}$.
\end{itemize}
While the duality \eqref{eq:Jacobi_transformation} is classical in the physics literature, its use to obtain asymptotic expansion at criticality seems to be less studied. Let us however mention the work by Kostov \cite{Ko-00} on the six-vertex model on a random lattice.

The main novelty in our complete asymptotics is the presence of logarithmic terms, see for instance \eqref{eq:series_expansion_log}. Such terms did not appear yet in the lattice walk literature, nor did they appear in the continuous setting, when deriving complete asymptotic expansions of the continuous heat kernel in cones \cite{ChFuRa-20}.  

Let us finally mention that from a viewpoint of potential theory, we provide examples of discrete heat kernel asymptotic estimates in two-dimensional cones. Recall that the discrete heat kernel in a cone $K$ is simply the probability $\mathbb P_x(S_n=y,\tau_K>n)$ that a random walk started at $x$ hits the point $y$ at time $n$ without exiting a given cone $K$ (the condition $\tau_K>n$, with $\tau_K$ denoting the exit time from $K$); this relation explains the title of this work.

\subsection{Earlier literature on asymptotic expansions for lattice walk models}
\label{sec:earlier_asymptotics}

\subsubsection*{The kernel method and algebraic solutions}

In lattice walk enumeration, a first source of asymptotic expansions is provided by the kernel method. In dimension $1$, the kernel method yields algebraic expressions for the generating function of the numbers of excursions (with given length, starting and ending points), starting from which it is possible to compute arbitrarily precise asymptotic expansions of the coefficients, using standard singularity analysis, see \cite{BaFl-02,BaWa-17} (although in these references, only one-term asymptotics are derived). In a few dimension $2$ and $3$ cases, the kernel method (or subtle variations of it, using the idea of half-orbit sums) also yields algebraic expressions for the generating function, for example for Kreweras' model, see \cite{FlHa-84,BM-02,BM-05,BMMi-10}, Gessel's model \cite{BM-16}, and \cite{BoBMKa-16} for some three-dimensional models. Again, in these cases, it is possible to deduce precise asymptotic expansions of the numbers of walks.

\subsubsection*{The kernel method and transcendental, D-finite solutions}

In dimension $2$ and more, the kernel method may also yield D-finite expressions for the generating functions as positive parts of rational functions, see \cite{BM-02,BMMi-10} for small steps in dimension $2$, \cite{BoBMKa-16} for small steps in dimension $3$ and \cite{BoBMMe-21} for large steps in dimension $2$. Let us notice that these ideas go beyond the case of quadrant (or octant) walks and also apply, for instance, in the framework of walks in the slit plane \cite{BM-01,BMSc-02} or the three quarter plane \cite{BMWa-20,BM-21,EP-22a,EP-22b}. 

There are several ways to pursue and to deduce from these expressions the asymptotics of the numbers of walks. In a few cases, it is possible to extract the coefficients in an explicit way (for instance, for Gouyou-Beauchamps walks), and then to deduce asymptotic expansions starting from these closed-form expressions. Many examples are provided in \cite{BM-02,BMMi-10,BoBMKa-16,CoMeMiRa-17,BoBMMe-21}. See \cite{Ne-23} for the complete asymptotic expansions in these cases.

Another possibility to continue is to use the modern theory of analytic combinatorics in several variables (ACSV), see for instance \cite{CoMeMiRa-17,MeWi-19} for examples of applications in the framework of lattice walks. In principle, using ACSV, one can deduce from these positive part expressions full asymptotic expansions for the numbers of walks. However, the applicability of the method is still restricted and the constants appearing in the prefactors of the asymptotic terms are not always easily computable.

\subsubsection*{The kernel method and non-D-finite solutions}
In a small number of cases, the (iterated) kernel method also applies to more singular models, associated to non-D-finite generating functions, see for instance \cite{MeMi-14}. Then it is possible to deduce some asymptotic estimates.

\subsubsection*{Weyl chambers}
Beyond the case of the quarter plane, Weyl chambers represent another class of cones, which is particularly popular (because of its links with non-intersecting paths and other probabilistic and physical models), and for which various asymptotic estimates exist. For related references, we refer to \cite{Bi-91,Bi-92} (with a strong emphasis on the link with representation theory), \cite{GeZe-92,EiKo-08,KoSc-10,DeWa-10,Fe-14} (in relation with reflectable random walks) and \cite{MeMi-16,MeWi-19,MiSi-20} (on highly symmetric lattice path models).

\subsubsection*{Guess and prove}
In a few works, the authors were able to guess the asymptotic behavior of various lattice path sequences, see \cite{BoKa-09} and \cite{BaKaYa-16}. There are also various experiments by Tony Guttmann \cite{Gu-19}, which are mainly not published.

\subsubsection*{Probabilistic and potential theory}
There is a simple and fundamental relation between numbers of walks and some probabilities. For example, there is equality between the total number of paths in a cone $K$ and the so-called survival probability $\mathbb P_x(\tau_K>n)$ (multiplied by some exponential factor); similarly, the number of excursions is directly related to the local probability $\mathbb P_x(S_n=y,\tau_K>n)$. Probabilistic local limit theorems (which by definition consist of the asymptotic derivation of the previous local probabilities) may therefore be directly translated into combinatorial estimates.

Using these ideas, one may first deduce the exponential growth of various numbers of walks. In particular, the exponential growth of the number of excursions (resp.\ total number of walks) is given in \cite{Ga-07} (resp.\ \cite{GaRa-16,JoMiYe-18}).

These rough estimates are refined in \cite{DeWa-15,Du-14}, where the authors obtain precise asymptotics of the numbers of excursions and of the total numbers of walks (the last result under the additional hypothesis of a drift equal to zero or directed to the vertex of the cone).

In the particular case of quadrant walk models with infinite group, these asymptotics are worked out in \cite{BoRaSa-14}. In particular, the critical exponent is shown to be non-rational in all these infinite group models.

Another fruitful approach is based on harmonic functions, which are deeply related to the asymptotics of lattice walks. Harmonic functions first appear in this context as the prefactors in the asymptotic estimates \cite{DeWa-15} of the number of excursions or the total number of walks. Moreover, their polynomial growth encodes the critical exponent of the number of excursions \cite{Va-99,DALaMu-16,Mu-19}. This approach has been generalized in \cite{ChFuRa-20,Ne-22}, where formal asymptotic expansions are derived in terms of polyharmonic functions.

\subsubsection*{Continuous heat kernel estimates}
Let $K$ be some cone in $\mathbb R^d$ and consider the Brownian motion $(B_t)_{t\geq0}$ killed at the boundary of $K$. Denote by $p(x,y;t)$ its transition density, that is the probability density function of the transition probability kernel
\begin{equation*}
     \mathbb P_x ( B_t \in dy,  \tau_K >t  ),
\end{equation*}
where $\tau_K$ is the first exit time of $K$. Recall the  well-known fact that $p(x,y;t)$ corresponds to the heat kernel, i.e., the fundamental solution of the heat equation on $K$ with Dirichlet boundary condition, see for instance \cite{De-87,BaSm-97}.
In \cite{ChFuRa-20}, the authors prove that the heat kernel  admits a complete asymptotic expansion in terms of continuous polyharmonic functions for the Laplacian.  See \cite{ArNaCr-83} for a general introduction to polyharmonic functions.

\subsubsection*{Boundary value problems and applications}
Following the pioneering works of Iasnogorodski and Fayolle \cite{FaIa-79}, Malyshev \cite{Ma-72}, see also \cite{FaIaMa-17}, Cohen and Boxma \cite{CoBo-83}, Cohen \cite{Co-92}, functional equations may be written for the generating functions of various probabilities (also for numbers of walks), and then boundary value problems may be deduced. This method results in contour integral expressions for the generating functions, on which one may try to apply singularity analysis, see \cite{FaRa-12}.

\subsubsection*{Other techniques}
Finally, let us conclude by mentioning a few other techniques. In dimension $3$, the numbers of walks may be connected to the computation of triangle eigenvalues \cite{BoPeRaTr-20,DaSa-20}. In dimension 2, there are also hypergeometric expressions \cite{BoChvHKaPe-17}, which lead to precise asymptotic estimates. One may consult \cite{Bo-HDR} for a nice and complete survey of lattice walk problems, which in particular contains many asymptotic results.

\subsection{Notation}
Given a model with a step-set $\mathcal{S}\subset \{-1,0,1\}^2$ and associated weights $\left(\omega_s\right)_{s\in\mathcal{S}}$, we will denote by $S(x,y)$ the step counting (Laurent) polynomial, given by 
\begin{equation}\label{eq:def_steppoly}
    S(x,y):=\sum_{(i,j)\in\mathcal{S}}\omega_{i,j}x^iy^j.
\end{equation}
One can then proceed to define the kernel (see e.g.~\cite{BMMi-10}) via 
\begin{equation}\label{eq:def_kernel}
    K(x,y)=xy\left[1-tS(x,y)\right].
\end{equation}
Associated to the kernel is an algebraic curve $\mathcal{C}$, which for non-singular models (by definition, singular models are step-sets whose all jumps lie in a linear half-plane)\ and small $t$ is elliptic, that is, of genus $1$, see e.g.~\cite{FaIaMa-17, HaSi-21,Dui-10}.  This curve is defined via 
\begin{equation}\label{eq:def_curve}
    \mathcal{C}:=\{(x,y)\in\overline{\C}\times\overline{\C}: K(x,y)=0\}.
\end{equation}
Here, $\overline{\C}$ is the projective closure of $\C$. For models with small steps, the path counting function \eqref{eq:def_series} satisfies the functional equation \cite{BMMi-10}
\begin{equation} K(x,y)Q(x,y)=K(x,0)Q(x,0)+K(0,y)Q(0,y)-K(0,0)Q(0,0)+xy.
    \label{eq:functional_eq}
\end{equation}



\section{An algebraic example: the Kreweras model}
\label{sec:Kreweras}
\subsection{Definition of the model}

By definition, the Kreweras model corresponds to the step-set $\{\leftarrow,\downarrow,\nearrow\}$, and the kernel 
\begin{equation*}
    K(x,y)=xy\left[1-t\left(xy+\frac{1}{x}+\frac{1}{y}\right)\right],
\end{equation*}
see Figure~\ref{fig:step_sets}. The functional equation \eqref{eq:functional_eq} for the path generating function \eqref{eq:def_series} takes the form 
\begin{equation}
\label{eq:functional_eq_Kreweras}
   K(x,y)Q(x,y) = xy - txQ(x,0) - tyQ(0,y).
\end{equation}
It is known that the generating function of this model admits the following expression:
\begin{equation}
\label{eq:expression_Kreweras_known}
    Q(x,0) = \frac{1}{xt}\left(\frac{1}{2t}-\frac{1}{x}-\left(\frac{1}{W}-\frac{1}{x}\right)\sqrt{1-xW^2}\right),
\end{equation}
with $W$ being the unique power series in $t$ solution to $W=t(2+W^3)$, see \cite{BM-02,BM-05,BMMi-10}.

\subsection{Parametrisation of the zero-set of the kernel}

We want to find a parametrisation of the elliptic curve $\mathcal{C}$ as defined in \eqref{eq:def_curve}, meaning we want to find functions $X(z)$ and $Y(z)$ meromorphic on $\mathbb C$ such that
\begin{equation}
\label{eq:def_parametrisation}
    \mathcal C = \{(X(z),Y(z)) : z\in\overline{\mathbb{C}}\}.
\end{equation}
Such a parametrisation has been obtained in \cite{FaIaMa-17,EP-20}. To state the result of \cite{EP-20}, we will utilise several classical properties of the theta function $\th(z)$; for now we mention the following three important properties, which can be immediately deduced from the series \eqref{eq:def_theta1} and \eqref{eq:def_theta2}:
\begin{equation}
\label{eq:elementary_property_theta}
    \th(\pi-z)=\th(z),\quad \th(-z)=-\th(z)\quad\text{and}\quad\th(z+\pi\tau)=-e^{-i\pi\tau-2iz}\th(z).
\end{equation}
First, define $\tau\in i\mathbb R_+$ in terms of $t\in(0,\frac{1}{3})$ as follows: 
\begin{equation}
\label{eq:relation_t_q_Kreweras}
   t =e^{-i\gamma/3}\frac{\th'(0|\tau)}{4i\th(\gamma|\tau)+6\th'(\gamma|\tau)},
\end{equation}
with $\gamma=\pi\frac{\tau}{3}$. 
The fact that $\tau$ is defined in terms of $t$ using an equation as~\eqref{eq:relation_t_q_Kreweras} will be shown in higher generality in Lemma~\ref{lem:unif_a-model}. Then setting
\begin{equation}
\label{eq:parametrisation_Kreweras}
    \left\{\begin{array}{rcl}
    X(z)&=&\displaystyle e^{-4i\gamma/3}\frac{\th(z|\tau)\th(z-\gamma|\tau)}{\th(z+\gamma|\tau)\th(z-2\gamma|\tau)},\medskip\\
    Y(z)&=&\displaystyle X(z+\gamma),
    \end{array}\right.
\end{equation}
Lemma~3 in \cite{EP-20} asserts that Equation~\eqref{eq:def_parametrisation} holds. In particular, we know that $X(z)$ has its only (double) pole at $z=-\gamma$ and $Y(z)$ at $z=\gamma$, see \eqref{eq:elementary_property_theta}. 

We can use \eqref{eq:relation_t_q_Kreweras} to write $t$ as a series in $q$, and as a consequence find an inverse, giving us 
\begin{equation} 
\label{eq:t_in_q}
   q=t^{9/2}+\frac{45}{2}t^{15/2}+\frac{4023}{8}t^{21/2}+\frac{184341}{16}t^{27/2}+\mathcal{O}\bigl(t^{33/2}\bigr).
\end{equation}
Notice that in \cite{EP-20}, the equivalent of \eqref{eq:t_in_q} is not exactly the same due to a slightly different choice in parametrisation: what is $q$ here corresponds to $q^{3/2}$ in \cite{EP-20}.

\subsection{Explicit expression for the generating function}

In order to obtain an expression for $Q(x,y)$, we will first find $Q(x,0)$ explicitly, and then make use of the functional equation \eqref{eq:functional_eq_Kreweras}. To that purpose, we will use an approach utilizing an invariant for this model \cite{BeBMRa-21,EP-20,EP-22a,EP-22b} and recall the proof of the following:
\begin{prop}[\cite{EP-20}]
\label{prop:expression_Kreweras}
We have
\begin{equation}\label{eq:Q(X(z),0)_formula}
   tX(z)Q(X(z),0)=\frac{1}{2t}-\frac{1}{X(z)}-J(z),
\end{equation} 
where (with $\gamma = \pi \frac{\tau}{3}$) 
\begin{equation}\label{eq:J(z)_formula}
J(z)=-e^{-4i\gamma/3}\frac{\th(2\gamma|\tau)\th'(0|\frac{\tau}{3})}{\th(\pi/2|\frac{\tau}{3})\th'(0|\tau)}\frac{\th(z+\pi/2|\frac{\tau}{3})}{\th(z|\frac{\tau}{3})}.
\end{equation}
\end{prop}


\begin{proof}
We make use of the fact that:\begin{align}
\label{eq:invariant_1}
    X(z)Y(z)+\frac{1}{X(z)}+\frac{1}{Y(z)}&=\frac{1}{t},\\
    \label{eq:invariant_2}
    tX(z)Q(X(z),0)+tY(z)Q(0,Y(z))&=X(z)Y(z).
     \end{align}
     Here, \eqref{eq:invariant_1} is due to the fact that $(X(z),Y(z))$ parametrises the kernel curve $\mathcal{C}$, see \eqref{eq:def_parametrisation}, and \eqref{eq:invariant_2} is a reformulation of the functional equation \eqref{eq:functional_eq_Kreweras}. Letting 
     \begin{equation}\label{eq:A(z)_definition}
         A(x):=\frac{1}{x}+xtQ(x,0)\quad \text{and}\quad
         B(y):=\frac{1}{y}+ytQ(0,y),
     \end{equation}
     we find that
     \begin{equation}\label{eq:J(z)_definition}
         J(z):=\frac{1}{2t}-A(X(z))=-\frac{1}{2t}+B(Y(z)).
     \end{equation}
     By symmetry of our model we could in fact conclude that in our case we have $A(x)=B(x)$, and hence $J(-z)=-J(z)$, but this is not necessary for our approach.
     
     Knowing that $J(z)$ can be expressed as a function of $X(z)$ and as a function of $Y(z)$ at the same time, it must have a range of symmetry properties:
     \begin{itemize}
         \item $J(\gamma-z)=J(z)$, because $J(z)$ is a function of $X(z)$;
         \item $J(-\gamma-z)=J(z)$, because $J(z)$ is a function of $Y(z)$;
         \item $J(2\gamma+z)=J(z)$ by the above;
         \item $J(z+\pi)=J(z)$ is inherited from this property of $\th$ functions, see \eqref{eq:elementary_property_theta}. 
     \end{itemize}
In particular, we know that $J$ is doubly periodic with periods $\pi$ and $2\gamma$. We thus construct as a candidate for $J(z)$ a function which has simple poles at $\{k\pi+\ell\gamma$: $(k,\ell)\in\Z^2\}$. This way
we obtain \eqref{eq:J(z)_formula}, by scaling with an appropriate constant and verifying that all poles of the difference between \eqref{eq:J(z)_formula} and \eqref{eq:J(z)_definition} vanish.
Using this together with \eqref{eq:A(z)_definition} and \eqref{eq:J(z)_definition}, we obtain \eqref{eq:Q(X(z),0)_formula}.
\end{proof}

\noindent From Prop.~\ref{prop:expression_Kreweras} we can deduce the representation~\eqref{eq:expression_Kreweras_known} of $Q(x,0)$ in the following manner: Defining $W:=X(\pi/2)=Y(\pi/2)$, the equation $W=t(2+W^3)$ follows from substituting $z=\frac{\pi}{2}$ into \eqref{eq:invariant_1}. Next, defining 
\[U(z):=\frac{\th(\gamma|\tau)}{\th(\frac{\pi}{2}+\gamma|\tau)}\frac{\th(z+\frac{\pi}{2}+\gamma|\tau)}{\th(z+\gamma|\tau)},\]
we have the equations
\begin{equation*}
    \frac{J(z)}{U(z)}+\frac{1}{X(z)}=\frac{1}{W}\quad \text{and}\quad 
    \frac{U(z)^2-1}{X(z)}=-W^2,
\end{equation*}
as in both cases the left-hand side is an elliptic function with no poles, where the constant value can be determined explicitly by setting $z$ to $\frac{\pi}{2}$ or $\frac{\pi}{2}-\gamma$. Combining these three equations yields \eqref{eq:expression_Kreweras_known}.

To keep computations short, in this section we will only give a representation for $Q(0,0)$ rather than all of $Q(x,y)$. Conveniently, $X(0)=0$, so it suffices to take the $z\to 0$ limit of our expression for $Q(X(z),0)$. Expanding both sides of \eqref{eq:Q(X(z),0)_formula} as series in $z$ then yields 
\begin{multline}
    \label{eq:Q(0,0)_formula}
    Q(0,0)=\frac{q^{2/3}\th(\gamma|\tau)^2}{\th'(0|\tau)^2}\times\\\left[\left(2
    +\frac{\th''\left(\frac{\pi}{2}|\frac{\tau}{3}\right)}{2\th\left(\frac{\pi}{2}|\frac{\tau}{3}\right)}-\frac{\th'''(0|\frac{\tau}{3})}{6\th'(0|\frac{\tau}{3})}+\frac{\th'''(0|\tau)}{6\th'(0|\tau)}\right)-\frac{6i\th'(\gamma|\tau)-\frac{1}{2}\th''(\gamma|\tau)}{\th(\gamma|\tau)}-\frac{4\th'(\gamma|\tau)^2}{\th(\gamma|\tau)^2}\right].
\end{multline}
Rewriting this as a series in $q$ and making use of \eqref{eq:t_in_q}, we obtain the generating function of Kreweras excursions (see \href{https://oeis.org/A006335}{A006335} in the OEIS)
\begin{equation*}
Q(0,0)=1+2t^3+16t^6+192t^9+\dots
\end{equation*}

\subsection{Effect of the Jacobi transformation}
As previously discussed, the Jacobi transformation \eqref{eq:Jacobi_transformation} involves a parameter $\hat{q}$, related to $q$ by $\log(q)\log(\hat{q})=\pi^2$.
Using \eqref{eq:Jacobi_identity} in \eqref{eq:relation_t_q_Kreweras},
we find that 
\begin{equation*}
    t=\frac{\th'(0,\hat{q})}{6\th'(\frac{\pi}{3},\hat{q})}.
\end{equation*}
This in turn can be used to express $\hat{q}$ as a series in $t$ around the critical point $\frac{1}{3}$ (the property that $\hat{q}=0$ actually corresponds to the critical value of $t$ turns out to be a general fact and will be proven in Lemma \ref{lemma:q_values}, see also the beginning of Section~\ref{subsec:effect_Jacobi}), which starts 
\begin{equation}
    \label{eq:qhat_in_t}
    \hat{q}=\frac{1}{\sqrt{3}}\left(\frac{1}{3}-t\right)^{1/2}-\frac{1}{\sqrt{3}}\left(\frac{1}{3}-t\right)^{3/2}+\frac{\sqrt{3}}{2}\left(\frac{1}{3}-t\right)^{5/2}-\frac{70}{27\sqrt{3}}\left(\frac{1}{3}-t\right)^{7/2}+\dots
\end{equation}
Next, we use the Jacobi identity \eqref{eq:Jacobi_identity} in order to rewrite \eqref{eq:Q(0,0)_formula} in terms of $\hat{q}$, which, after some simplifications, yields 
\begin{multline}\label{eq:q00_critical_formula}
    Q(0,0)=\frac{\th\left(\frac{\pi}{3},\hat{q}\right)}{\th'(0,\hat{q})^2}\Bigg[-\frac{9\th\left(\frac{2\pi}{3},\hat{q}\right)}{2}-\frac{\th''\left(\frac{\pi}{3},\hat{q}\right)}{2}+\frac{\th\left(\frac{\pi}{3},\hat{q}\right)\th'''\left(0,\hat{q}\right)}{6\th'(0,\hat{q})}-\frac{4\th'\left(\frac{\pi}{3},\hat{q}\right)}{\th\left(\frac{\pi}{3},\hat{q}\right)}\\-\frac{3\th\left(\frac{2}{3},\hat{q}\right)\th'''\left(0,\hat{q}^3\right)}{2\th'\left(0,\hat{q}^3\right)}+\frac{9\th\left(\frac{2\pi}{3},\hat{q}\right)}{\th\left(\frac{3i}{2}\log\hat{q},\hat{q}^3\right)}\left(\frac{1}{2}\th''\left(\frac{3i}{2}\log\hat{q},\hat{q}^3\right)-i\th'\left(\frac{3i}{2}\log\hat{q},\hat{q}^3\right)\right)\Bigg].
\end{multline}
Note in particular that the terms of the form $\th^{(k)}\left(\frac{3i}{2}\log\hat{q},\hat{q}^3\right)$ are not quite as unwieldy as they appear, since the dependency of $\th$ (or its derivatives) on the first component is essentially given by trigonometric functions, and we have for instance \begin{equation*}
    \sin\left(\frac{3i}{2}\log\hat{q}\right)
    =\frac{1}{2i}\left(\hat{q}^{-3/2}-\hat{q}^{3/2}\right).
\end{equation*}
That a simplification of this form works is to be expected, seeing as we already know this model to be algebraic, thus all logarithms must vanish \cite{FlHa-84,BM-02,BM-05,BMMi-10}. In terms of the computation, this relies heavily on the relation $\gamma=\pi\frac{\tau}{3}$ in the parametrisation \eqref{eq:parametrisation_Kreweras}. For the model we study in Section~\ref{sec:a-model}, there is no similar relation between $\gamma$ and $\tau$, and thus there is no intuitive reason why the logarithms should disappear. 

\subsection{Series expansion around the critical point}

Lastly, all we need to do is substitute \eqref{eq:qhat_in_t} into \eqref{eq:q00_critical_formula} in order to find a local expansion at the critical point $t=\frac{1}{3}$, starting \begin{equation*}
    Q(0,0)=\frac{9}{8}-\frac{27}{8}\left(\frac{1}{3}-t\right)-9\sqrt{3}\left(\frac{1}{3}-t\right)^{3/2}-\frac{189}{4} \left(\frac{1}{3}-t\right)^2+72\sqrt{3}\left(\frac{1}{3}-t\right)^{5/2}\pm\dots
\end{equation*}
One can easily verify that this coincides with the explicit expression for $Q(0,0)$ given in \cite{FlHa-84,BM-02,BM-05,BMMi-10}.

\section{An infinite group model}
\label{sec:a-model}

\noindent We follow the exact same exposition as in Section~\ref{sec:Kreweras}. 
\subsection{Definition of the model}\label{sec:def_a-model}

We consider the problem of quadrant walks with $\sf N$, $\sf E$, $\sf S$, $\sf W$ and $\sf NE$ steps with a weight $a>0$ for each $\sf NE$ step, see Figure~\ref{fig:step_sets}. The generating function for this model was shown to be D-algebraic by Bernardi, Bousquet-M\'elou and Raschel \cite{BeBMRa-21}. 
As in Section \ref{sec:Kreweras}, we start with the functional equation \eqref{eq:functional_eq}, which takes the form
\begin{equation*}
   \frac{K(x,y)}{xyt} Q(x,y)=R(x,y),
\end{equation*}
where $K(x,y)$ is the kernel as in \eqref{eq:def_kernel}, and the term
\begin{equation*}
   R(x,y)=\frac{1}{t}-\frac{1}{x}Q(0,y)-\frac{1}{y}Q(x,0)
\end{equation*}
is called the remainder. The benefit of writing the equation in this form is that we know that if $K(x,y)=0$ then we also have $R(x,y)=0$, as long as the series all converge (alternatively one could write $x$ as a formal power series of $y$ satisfying this equation, but we choose to take the analytic approach in this work). We note that for $t$ sufficiently small (i.e., $\vert t\vert<1/( a+4)$), the series $Q(x,y)$ converges in the domain where $\vert x\vert,\vert y\vert<1$.

\subsection{Parametrisation of the zero-set of the kernel}
In this subsection, we deduce a parametrisation of the kernel curve using results from \cite{EP-22b} (which follows \cite{FaIa-79,FaIaMa-17,Ra-12,DrRa-19}) for general weighted step-sets, then specialising these results to our step-set. We consider the curve $\mathcal C=\{(x,y)\in\overline{\C}\times\overline{\C}:K(x,y)=0\}$ as in \eqref{eq:def_curve}. Under the assumption that our step-set is non-singular and that 
\begin{equation}
    \label{eq:as_tt}
    0<t<\frac{1}{S(1,1)}=\frac{1}{a+4},
\end{equation}
we have the following lemmas (see \cite[Lem.~2.3]{EP-22b}):
\begin{lem}\label{lem:param}
There are meromorphic functions $X,Y:\mathbb{C}\to\mathbb{C}\cup\{\infty\}$ which parametrise $\mathcal C$, that is
\begin{equation*}
   \mathcal C=\{(X(z),Y(z)):z\in\mathbb{C}\},
\end{equation*}
and numbers $\gamma,\tau\in i\mathbb{R}$ with $\Im(\pi\tau)>\Im(2\gamma)>0$ satisfying the following conditions:
\begin{itemize}
\item $K(X(z),Y(z))=0$;
\item $X(z)=X(z+\pi)=X(z+\pi\tau)=X(-\gamma-z)$;
\item $Y(z)=Y(z+\pi)=Y(z+\pi\tau)=Y(\gamma-z)$;
\item $\vert X(-\frac{\gamma}{2})\vert ,\vert Y(\frac{\gamma}{2})\vert <1$;
\item Counting with multiplicity, the functions $X(z)$ and $Y(z)$ each contain two poles and two roots in each fundamental domain $\{z_{c}+r_{1}\pi+r_{2}\pi\tau:r_{1},r_{2}\in[0,1)\}$.
\end{itemize}
\end{lem}

\begin{figure}[ht]
\centering
   \includegraphics[scale=0.5]{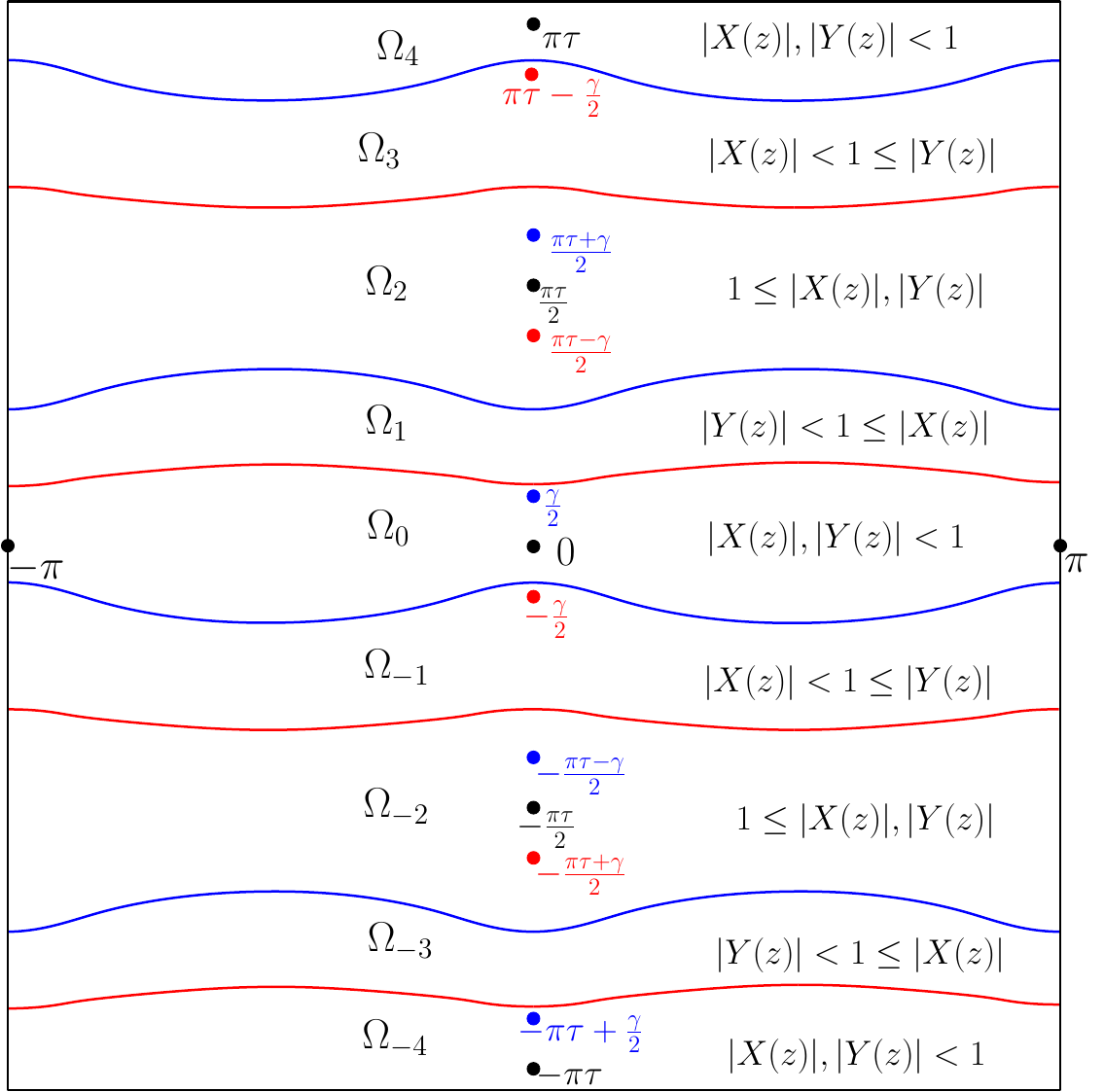}
   \caption{The complex plane partitioned into regions $\Omega_{j}$. For $z$ on the blue lines, $\vert Y(z)\vert=1$, while on the red lines $\vert X(z)\vert=1$.}
   \label{fig:Omega_chart}
\end{figure}
\noindent The following is \cite[Lem.~2.5]{EP-22b}, under the assumption~\eqref{eq:as_tt}.
\begin{lem}\label{lem:Omega}
The complex plane can be partitioned into simply connected regions $\{\Omega_{s}\}_{s\in\mathbb{Z}}$ as in Figure~\ref{fig:Omega_chart}, satisfying
\begin{align*}
\bigcup_{s\in\mathbb{Z}}\Omega_{4s}\cup\Omega_{4s+1}&=\{z\in\mathbb{C}:\vert Y(z)\vert <1\},\\
\bigcup_{s\in\mathbb{Z}}\Omega_{4s-2}\cup\Omega_{4s-1}&=\{z\in\mathbb{C}:\vert Y(z)\vert \geq1\},\\
\bigcup_{s\in\mathbb{Z}}\Omega_{4s-1}\cup\Omega_{4s}&=\{z\in\mathbb{C}:\vert X(z)\vert <1\},\\
\bigcup_{s\in\mathbb{Z}}\Omega_{4s+1}\cup\Omega_{4s+2}&=\{z\in\mathbb{C}:\vert X(z)\vert \geq1\}.
\end{align*}
Moreover, the equations
\begin{align*}
\pi+\Omega_{s}&=\Omega_{s},\\
s\pi\tau+\gamma-\Omega_{2s}\cup\Omega_{2s+1}&=\Omega_{2s}\cup\Omega_{2s+1}\supset \frac{s\pi\tau+\gamma}{2}+\mathbb{R},\\
s\pi\tau-\gamma-\Omega_{2s}\cup\Omega_{2s-1}&=\Omega_{2s}\cup\Omega_{2s-1}\supset\frac{s\pi\tau-\gamma}{2}+\mathbb{R},
\end{align*}
hold for each $s\in\mathbb{Z}$.
\end{lem}
\noindent The third result that we will need involves the Jacobi theta function as in \eqref{eq:def_theta1} and is~\cite[Prop.~2.6]{EP-22b}. 
\begin{prop}\label{prop:th_param}
There are some $\alpha\in\Omega_{0}\cup\Omega_{-1}$, $\beta\in\Omega_{0}\cup\Omega_{1}$, $\delta\in\Omega_{1}\cup\Omega_{2}$, $\epsilon\in\Omega_{-2}\cup\Omega_{-1}$ and $x_{c},y_{c}\in\mathbb{C}\setminus\{0\}$ satisfying
\begin{equation*}
\left\{\begin{array}{rcl}
X(z)&=&\displaystyle x_{c}\frac{\th(z-\alpha|\tau)\th(z+\gamma+\alpha|\tau)}{\th(z-\delta|\tau)\th(z+\gamma+\delta|\tau)},\medskip\\
Y(z)&=&\displaystyle y_{c}\frac{\th(z-\beta|\tau)\th(z-\gamma+\beta|\tau)}{\th(z-\epsilon|\tau)\th(z-\gamma+\epsilon|\tau)}.\end{array}\right.
\end{equation*}
\end{prop}

\noindent In the following result, we specialise the above parametrisation to our weighted step-set and thus describe the zero-set of the kernel, meaning the set $\mathcal C$ as introduced in \eqref{eq:def_curve}:
\begin{lem}
\label{lem:unif_a-model}
Given a step-set as defined in Section~\ref{sec:def_a-model}, that is, with counting polynomial $S(x,y)=x+y+\frac{1}{x}+\frac{1}{y}+axy$, and assuming~\eqref{eq:as_tt},
there exists a triple $(c,\gamma,\tau)\in\mathbb{C}^3$ with $\gamma,\tau\in i\mathbb{R}$, $0<\Im(2\gamma)<\Im(\pi\tau)$, satisfying the following equations:
\begin{align}
ac^3&=-\frac{\th(4\gamma|\tau)}{\th(2\gamma|\tau)}\label{eq:ac3},\\
c^2&=\frac{\th(3\gamma|\tau)}{\th(\gamma|\tau)}\label{eq:c2},\\
\frac{c}{t}&=4\frac{\th'(\gamma|\tau)\th(2\gamma|\tau)}{\th'(0|\tau)\th(\gamma|\tau)}-2\frac{\th'(2\gamma|\tau)}{\th'(0|\tau)}.
\label{eq:c_on_t}
\end{align}
With the above notation, the functions $X(z)$ and $Y(z)$ given by
\begin{equation*}
\left\{\begin{array}{rcl}
X(z)&=&\displaystyle c\frac{\th(z|\tau)\th(z+\gamma|\tau)}{\th(z-\gamma|\tau)\th(z+2\gamma|\tau)},\medskip\\
Y(z)&=&\displaystyle c\frac{\th(z|\tau)\th(z-\gamma|\tau)}{\th(z+\gamma|\tau)\th(z-2\gamma|\tau)}\,=\,X(-z),
\end{array}\right.
\end{equation*}
satisfy the conditions of Lemmas~\ref{lem:param} and \ref{lem:Omega}.
\end{lem}
\begin{proof} 
From the definition of $X(z)$ and $Y(z)$ given in \cite[App.~A]{EP-22b}, the fact that the step-set is symmetric implies that $X(z)=Y(-z)$. Moreover, $\Omega_{0}\cup\Omega_{1}$ is the connected component of $\{z\in\mathbb{C}:|Y(z)|<1\}$ containing $\frac{\gamma}{2}+\mathbb{R}$, while $\Omega_{0}\cup\Omega_{-1}$ is the connected component of $\{z\in\mathbb{C}:|X(z)|<1\}$ containing $-\frac{\gamma}{2}+\mathbb{R}$. Combining with $X(z)=Y(-z)$, this implies that $-\Omega_{0}\cup\Omega_{1}=\Omega_{-1}\cup\Omega_{0}$. Considering the intersection $\Omega_{0}=\left(\Omega_{0}\cup\Omega_{-1}\right)\cap\left(\Omega_{0}\cup\Omega_{1}\right)$, we see that $-\Omega_{0}=\Omega_{0}$.

Now, by the definition \eqref{eq:def_curve} of $\mathcal{C}$, we have $(0,0),(\infty,0),(0,\infty)\in \mathcal{C}$.

From Lemma~\ref{lem:param}, we know that $X(z)$ and $Y(z)$ each contain two roots and two poles in each fundamental domain. Consider the fundamental domain
\begin{equation*}
   F=\{z\in\Omega_{-1}\cup\Omega_{0}\cup\Omega_{1}\cup\Omega_{2}:\Re(z)\in[0,\pi)\},
\end{equation*}
and let $\alpha\in F$ and $\delta\in F$ be a root and pole of $X(z)$, respectively, which are both roots of $Y(z)$. From Lemma \ref{lem:Omega}, we must have $\alpha\in \Omega_{0}$ and $\delta\in\Omega_{1}$. Since these are distinct, $Y(z)$ has no other roots in $F$, and so the complete set of roots of $Y(z)$ in $\Omega_{0}\cup\Omega_{1}$ is $\alpha+\pi\mathbb{Z}\cup \delta+\pi\mathbb{Z}$. Note that $-\alpha\in\Omega_{0}$ is also a root of $Y(z)$, so we must have $-\alpha\in\alpha+\pi\mathbb{Z}$. In fact, since $\alpha\in F$, we have $\alpha=\frac{\pi}{2}$ or $\alpha=0$. We will start by considering the case where $\alpha=0$. Then $0$ is a root of $Y(z)$, so, since $Y(\gamma-z)=Y(z)$, the value $\gamma\in\Omega_{0}\cup\Omega_{1}$ is also a root of $Y(z)$. Since $\Re(\gamma)=0$, we have $\gamma\in F$, so $\delta=\gamma$. Recall also that $\delta$ is a pole of $X(z)$, so $-\gamma=-\delta$ is a pole of $Y(z)$. This implies that the function $\tilde{Y}(z)$ defined by
\begin{equation*}
   \tilde{Y}(z):=Y(z)\frac{\th(z+\gamma|\tau)\th(z-2\gamma|\tau)}{\th(z|\tau)\th(z-\gamma|\tau)}
\end{equation*}
has at most a single pole in each fundamental domain, at the pole of $Y(z)$ other than $-\gamma$. But $\tilde{Y}(z)$ is an elliptic function with periods $\pi$ and $\pi\tau$, so it cannot have only a single pole in each fundamental domain \cite[p.~8]{Akh-90}. Therefore it must have no poles, and is therefore a constant function.

Now write $\tilde{Y}(z)=c$ where $c\in\mathbb{C}$. Note $c\neq 0$ as $Y(z)$ is not the zero function. Then
\begin{equation*}
   Y(z)=c\frac{\th(z|\tau)\th(z-\gamma|\tau)}{\th(z+\gamma|\tau)\th(z-2\gamma|\tau)}
\end{equation*}
and
\begin{equation*}
   X(z)=Y(-z)=c\frac{\th(z|\tau)\th(z+\gamma|\tau)}{\th(z-\gamma|\tau)\th(z+2\gamma|\tau)}.
\end{equation*}
Equations \eqref{eq:ac3}, \eqref{eq:c2} and \eqref{eq:c_on_t} follow from considering the equation $K(X(z),Y(z))=0$ at $z=2\gamma$, $\gamma$ and $0$.

Finally we will discuss the case $\alpha=\frac{\pi}{2}$ instead of $\alpha=0$. The functions $\hat{X}(z):=X(z+\pi/2)$ and $\hat{Y}(z):=Y(z+\pi/2)$ satisfy $\hat{X}(0)=\hat{Y}(0)$, so we can apply the rest of the proof to $\hat{X}(z)$ and $\hat{Y}(z)$, to show that they have the required properties. By carefully analysing the definition of $X(z)$ and $Y(z)$ in \cite[App.~A]{EP-22b}, one can prove that the $\alpha=\frac{\pi}{2}$ case actually never occurs; this is however unnecessary for the proof.
\end{proof}
\noindent\textbf{Remark:} The fact that $\frac{\th(3\gamma\vert \tau)}{\th(\gamma\vert \tau)}>0>\frac{\th(4\gamma\vert \tau)}{\th(2\gamma\vert \tau)}$ along with $\frac{\gamma}{\pi\tau}\in(0,\frac{1}{2})$ implies that $\frac{\gamma}{\pi\tau}\in(\frac{1}{4},\frac{1}{3})$.
Observe that the above uniformization is very similar to that \eqref{eq:parametrisation_Kreweras} of Kreweras' model. Unlike in the latter case, however, where we had (after rescaling $\tau$) $\alpha=3\tau$, there is no obvious relation between $\gamma$ and $\tau$ here.

\subsection{Explicit expression for the generating function}

As a consequence of the above definitions, we can define holomorphic functions $Q_{1}:\Omega_{-1}\cup\Omega_0\to\mathbb{C}$ and $Q_{2}:\Omega_0\cup\Omega_1\to\mathbb{C}$ by
\begin{equation*}
   Q_{1}(z)=Q(X(z),0)\qquad\text{and}\qquad Q_{2}(z)=Q(0,Y(z)).
\end{equation*}
Moreover, by symmetry, $Q_{1}(z)=Q_{1}(\gamma-z)$ and $Q_{2}(z)=Q_{2}(-\gamma-z)$. 

The next step is to find a function of $X(z)$ which is equal to a function of $Y(z)$ for $z\in\Omega_0$ using the equations $K(X(z),Y(z))=R(X(z),Y(z))=0$. Note that, given \eqref{eq:functional_eq}, this is equivalent to finding a decoupling function in the sense of~\cite[4.2]{BeBMRa-21}.
In this case, this is immediate as combining the two equations yields
\begin{equation*} atX(z)Q(X(z),0)+atY(z)Q(0,Y(z))=aX(z)Y(z)=-X(z)-Y(z)-\frac{1}{X(z)}-\frac{1}{Y(z)}+\frac{1}{t},
\end{equation*}
so we can now define a meromorphic function $J(z)$ on $\Omega_{-1}\cup\Omega_0\cup\Omega_1$ by writing
\begin{equation}
\label{eq:Jinxz}
J(z) := \left\{\begin{array}{ll}
\displaystyle \hspace{6.4mm}\frac{1}{2t}-atX(z)Q_{1}(z)-X(z)-\frac{1}{X(z)}& \text{for $z\in\Omega_{-1}\cup\Omega_0$,}\medskip \\
\displaystyle -\left(\frac{1}{2t}-atY(z)Q_{2}(z)-Y(z)-\frac{1}{Y(z)}\right)&\text{for $z\in\Omega_0\cup\Omega_1$},
\end{array}\right.
\end{equation}
as these expressions are equal on $\Omega_0$. Moreover, $J(z)$ satisfies $J(\gamma-z)=J(z)=J(z+\pi)$ for $z\in\Omega_{-1}\cup\Omega_0$ and $J(-\gamma-z)=J(z)$ for $z\in \Omega_0\cup\Omega_1$, so $J(z-\gamma)=J(z+\gamma)$ for $z\in \Omega_0$. We can use this to extend $J(z)$ to a meromorphic function on $\mathbb{C}$, which satisfies
\begin{equation*}
   J(\gamma-z)=J(z),\quad J(-\gamma-z)=J(z)\quad\text{and}\quad J(z+\pi)=J(z).
\end{equation*}
This implies that $J$ is doubly periodic, with periods $\pi$ and $2\gamma$, which will allow us to determine it exactly. We also note that by symmetry in $x$ and $y$, we have $J(z)+J(-z)=0$. Solving this exactly yields the following result, analogous to Proposition~\ref{prop:expression_Kreweras} for the Kreweras model:
\begin{prop}
\label{prop:expression_Q(x,0)_a-model}
Let $c,\gamma,\tau$ be defined as in Lemma~\ref{lem:unif_a-model}. The function $J$ in \eqref{eq:Jinxz} is given by
\begin{equation*}J(z)=-\frac{\th'(0|\frac{\gamma}{\pi})\th(2\gamma|\tau)\th(z+\frac{\pi}{2}|\frac{\gamma}{\pi})}{c\th(\frac{\pi}{2}|\frac{\gamma}{\pi})\th'(0|\tau)\th(z|\frac{\gamma}{\pi})}.\end{equation*}
\end{prop}
\begin{proof}
Define $I(z)$ by
\begin{equation*}
   I(z):=J(z)+\frac{\th'(0|\frac{\gamma}{\pi})\th(2\gamma|\tau)\th(z+\frac{\pi}{2}|\frac{\gamma}{\pi})}{c\th(\frac{\pi}{2}|\frac{\gamma}{\pi})\th'(0|\tau)\th(z|\frac{\gamma}{\pi})}.
\end{equation*}
Then it suffices to show that $I(z)=0$. Using properties of $\th$ (see \eqref{eq:elementary_property_theta}), we observe that $I(z)$ satisfies the same transformations as $J(z)$, namely $I(z)=-I(-z)=I(\gamma-z)=I(z+\pi)$. Moreover, by \eqref{eq:Jinxz}, the poles of $J(z)$ in $\Omega_{-1}\cup\Omega_0$ occur precisely at the points $\pi n$ and $\gamma+\pi n$ for $n\in\mathbb{Z}$. By the definition, $I(z)$ has no other poles in $\Omega_{-1}\cup\Omega_0$, and taking $z\to 0$, we see that $0$ is not a pole of $I(z)$. Hence, the transformations $I(z)=I(\gamma-z)=I(z+\pi)$ imply that $I(z)$ must be holomorphic on $\Omega_{-1}\cup\Omega_0$. Moreover, since $I(z)=-I(\gamma+z)$, this implies that $I(z)$ is holomorphic on $\mathbb{C}$, so it is a constant function. Finally we find that this constant $k$ is $0$ from $k=I(z)=-I(-z)=-k$. 
\end{proof}
\noindent Proposition~\ref{prop:expression_Q(x,0)_a-model} combined with \eqref{eq:Jinxz} gives an explicit expression for $Q(x,0)$, starting from which we can extract the exact form of $Q(0,0)$. For convenience we will write $\th(z)=\th(z|\tau)$ and $\thh(z)=\th(z|\frac{\gamma}{\pi})$.
Analysing $J(z)$ as $z\to0$ yields the equation
\begin{multline*}
1+atQ(0,0)=\\\frac{\th(2\gamma)^2}{c^2\th'(0)^2}\left(
\frac{1}{2}\frac{\th'(\gamma)\th'(2\gamma)}{\th(\gamma)\th(2\gamma)}-2\frac{\th'(\gamma)^2}{\th(\gamma)^2}-\frac{1}{2}\frac{\th''(2\gamma)}{\th(2\gamma)}+\frac{1}{6}\frac{\th'''(0)}{\th'(0)}+\frac{1}{2}\frac{\thh''(\frac{\pi}{2})}{\thh(\frac{\pi}{2})}-\frac{1}{6}\frac{\thh'''(0)}{\thh'(0)}
\right).
\end{multline*}


We now describe how the coefficients of the series $Q(0,0)$ can be extracted from the solution above. Writing $s=e^{i\gamma}$, the right-hand side of the equation
\begin{equation}
\label{eq:relation_a_q}
a^{2}=\frac{\th(\gamma|\tau)^3\th(4\gamma|\tau)^2}{\th(2\gamma|\tau)^2\th(3\gamma|\tau)^3}
\end{equation}
is a series in $q$ and $s$, while the left-hand side is constant. We can use this to write $q$ as a series in $s$, which starts 
\begin{equation*}
q=a^{1/2}s^{7/2}+\left(\frac{1}{2\sqrt{a}}-\frac{3}{4}a^{3/2}\right)s^{9/2}+\mathcal O(s^{11/2}).
\end{equation*}
We can then write $t$ (and $c$) as series in $s$ using \eqref{eq:ac3}, \eqref{eq:c2} and \eqref{eq:c_on_t}.
Consequently our expression for $Q(0,0)$ can be expanded as a series in $s$ and therefore $t$.

\subsection{Effect of the Jacobi transformation}
\label{subsec:effect_Jacobi}

Before analysing the Jacobi transformation \eqref{eq:Jacobi_transformation}, let us define the critical point of the model. As shown in \cite{FaRa-10,FaRa-12,BoRaSa-14,DeWa-15}, the critical point $t_c>0$ may be defined as the smallest positive singularity of the series $Q(0,0)$. It can be further characterized as the exponential growth of the coefficients of $Q(0,0)$, meaning that (up to a polynomial correction) $[t^n]Q(0,0)\sim 1/t_c^n$ (see \eqref{eq:asymptotics_Qx0} for a more precise statement). Finally, it is also the smallest value of $t>0$ such that the Riemann surface $\mathcal C$ has genus $0$.

What is essential for applying the Jacobi transformation is the fact that the critical point corresponds to $q=1$ (and consequently $\hat{q}=0$, whereas $\hat{q}=1$, or $q=0$, is equivalent to the regime $t=0$). Additionally, we will also need to know that the only singularity of $Q(0,0)$ lies on the positive real axis. These two facts will be shown in the following Lemmas~\ref{lemma:q_values} and~\ref{lem:real_sing}, which are true for more general models than considered here. They will be proven in Appendix~\ref{app:lemma:q_values}. 
\begin{lem}
\label{lemma:q_values}
Writing $q$ as a function of $t$, we have \begin{enumerate}
    \item $0<q(t)<1$ for $0<t<t_c$,
    \item $q(t)$ as a function of $t$ is continuous on $(0,t_c)$,
    \item $\lim_{t\to 0}q(t)=0$,
    \item $\lim_{t\to t_c}q(t)=1$.
\end{enumerate}
\end{lem}

\begin{lem}
\label{lem:real_sing}
Let $\mathcal S$ be a non-singular, weighted step-set, and let $Q(x,y)$ be the generating function \eqref{eq:def_series} for walks in the quadrant using this step-set. Define the {\em period} of the model to be the maximum value $k$ such that $Q(0,0)=Q(0,0;t)\in\mathbb{R}[t^k]$. If $r$ is the radius of convergence of $Q(0,0)$, then the singularities of $Q(0,0)$ on the radius of convergence are precisely the points $re^{\frac{2\pi i j}{k}}$ for $j=0,1,\ldots,k-1$. The same result holds for the generating function $[x^a][y^b]Q(x,y)$ for any $a,b$ for which this generating function is non-zero.
\end{lem}

\noindent Applying Lemma~\ref{lem:real_sing} to our case, since $a>0$, the period $k$ is $1$, so the only singularity on the radius of convergence is on the positive real line. Moreover, Lemma \ref{lemma:q_values} implies that $q$ and $\tau$ have no singularity for $t$ in the interval $[0,t_{c})$. In addition, $\gamma$ in Lemma~\ref{lem:unif_a-model} is analytic as well on $[0,t_{c})$; this follows from the expression of $\gamma$ in terms of two periods $\omega_1,\omega_3$ given in \cite[App.~A]{EP-22b}, and from the analytic behavior of these periods shown in \cite[Sec.~7.4]{KuRa-12}. Together these imply that if $Q(0,0)$ has a singularity at $t_c$, then it is the unique singularity within the radius of convergence, so the asymptotic form of the coefficients is uniquely determined by the behaviour at this point. The same holds for all coefficients of $Q(x,0)$ in its series expansion at $x=0$. Again by Lemma~\ref{lemma:q_values}, the point $t=t_c$ corresponds to $\hat{q}=0$. For this reason we proceed by analysing the parameters at $\hat{q}=0$.

It is convenient to parametrise $a$ using the unique $k\in(0,1)$ satisfying
\begin{equation}a=\frac{1-k^{2}}{k^{3}}.\label{eq:expression_a_k}\end{equation}
Writing $\beta=\frac{i}{\pi}\log(\hat{q})\gamma$, the equation \eqref{eq:relation_a_q} relating $q$ and $a$ becomes
\begin{equation*}
   a^{2}=\frac{\th(\beta,\hat{q})^{3}\th(4\beta,\hat{q})^{2}}{\th(2\beta,\hat{q})^{2}\th(3\beta,\hat{q})^{3}}.
\end{equation*}
This equation allows $2\cos(2\beta)$ to be written as a series in $\hat{q}$, with initial terms
\[2\cos(2\beta)=k^2-1-k^2(2k^2-1)(k^2+1)(k^2-3)(k^2-1)\hat{q}^{2}+O(\hat{q}^4).\]
This allows us to write $\beta$ itself as a series in $\hat{q}$:
\[\beta=\beta_{0}+\beta_{1}\hat{q}^{2}+\beta_{2}\hat{q}^{4}+O(\hat{q}^6),\]
where the constant term $\beta_{0}$ is given by
\begin{equation}\label{eq:relation_k_beta}
   \beta_{0}=\frac{1}{2}\cos^{-1}\left(\frac{k^{2}-1}{2}\right),
\end{equation}
while $\beta_{1}$ is given by
\begin{equation*}
   \beta_{1}=\frac{1}{2}k^2(2k^2-1)(1-k^2)\sqrt{(1+k^2)(3-k^2)}.
\end{equation*}
It follows that $t$ is also a series in $\hat{q}$, given by
\begin{equation*}
   \frac{1}{t}=-a-\frac{1}{a}-\frac{1}{a}\frac{\th(\beta,\hat{q})^2\th(5\beta,\hat{q})}{\th(3\beta,\hat{q})^3}.
\end{equation*}
The initial terms are then
\[t=\frac{k}{k^2+3}\left(1-\frac{(k^2+1)^2(3-k^2)^3}{k^2+3}\hat{q}^{2}\right)+O(\hat{q}^4).\]
In particular, the critical point $t_{c}$ is given by
\begin{equation*}
   t_{c}=\frac{k}{k^{2}+3},
\end{equation*}
so $t_{c}$ is given by an algebraic function of $a$. Taking the inverse of the series above yields
\[\hat{q}^{2}=\frac{k^2+3}{(k^2+1)^2(3-k^2)^3}\Bigl(1-\frac{t}{t_{c}}\Bigr)+O\left(\Bigl(1-\frac{t}{t_{c}}\Bigr)^2\right).\]
Notice that the parameter $\beta_0$ is connected to another relevant parameter $\theta$ introduced in \cite{DeWa-15,BoRaSa-14}. Based on \cite[Ex.~2]{DeWa-15}, $\theta$ is computed in \cite{BoRaSa-14} to describe the asymptotic behavior of walks in two-dimensional cones. More precisely, let $S(x,y)$ be the step polynomial of the model as in \eqref{eq:def_steppoly}. For non-degenerate models, there exists a unique point $(x_0,y_0)\in(0,\infty)^2$ such that $\frac{\partial S}{\partial x}(x_0,y_0)=\frac{\partial S}{\partial y}(x_0,y_0)=0$. Then the parameter $\theta$ can be defined as follows
\begin{equation*}
    \theta = \arccos \left(-\frac{\frac{\partial^2 S}{\partial x\partial y}(x_0,y_0)}{\sqrt{\frac{\partial^2 S}{\partial x^2}(x_0,y_0)\cdot \frac{\partial^2 S}{\partial y^2}(x_0,y_0)}}\right).
\end{equation*}
Standard computations give that $x_0=y_0$ are both solutions to the equation $a x_0^3=1-x_0^2$, hence with our notation \eqref{eq:expression_a_k}, we have $x_0=y_0=k$. Accordingly, $\theta = \arccos (\frac{k^2-1}{2}) = 2\beta_0$. 

\subsection{Series expansion of \texorpdfstring{$Q(x,0)$}{the series} at the critical point}

In order to understand the asymptotics of the coefficients of $Q(x,0)$, we will write $Q(x,0)$ as a series in $x$ and $\hat{q}$. Recall that an expression for $Q(x,0)$ was given parametrically by $J(z)$ and $X(z)$, see Proposition~\ref{prop:expression_Q(x,0)_a-model}. It is important to notice that in the previously cited proposition, it is assumed that $t<\frac{1}{S(1,1)}$, see \eqref{eq:as_tt}, while we now want to work with $t$ close to $t_c>\frac{1}{S(1,1)}$. We observe that the identity in Proposition~\ref{prop:expression_Q(x,0)_a-model}, giving an expression for the generating function in terms of theta functions, can be readily extended from $t\in (0,\frac{1}{S(1,1)})$ to $t\in (0,t_c)$ by analytic continuation, the two sides of the identities being actually analytic in that bigger domain.

Using then the Jacobi identity \eqref{eq:Jacobi_identity} on the expressions for $X(z)$ and $J(z)$ yields (see Lemma~\ref{lem:unif_a-model} and Proposition~\ref{prop:expression_Q(x,0)_a-model})
\begin{equation}\label{eq:Xz_in_qhat}
   X(z)=d\frac{\th(-z\hat{\tau},\hat{q})\th(-\beta-z\hat{\tau},\hat{q})}{\th(\beta-z\hat{\tau},\hat{q})\th(-2\beta-z\hat{\tau},\hat{q})}
\end{equation}
and
\begin{equation}\label{eq:Jz_in_qhat}
   J(z)=-e^{\frac{i\pi\hat{\tau}z}{\beta}}\frac{\pi}{\beta d}\frac{\th(2\beta,\hat{q})\th(-\pi\hat{\tau}\frac{\pi+2z}{2\beta},\hat{q}^{\frac{\pi}{\beta}})\th'(0,\hat{q}^{\frac{\pi}{\beta}})}{\th'(0,\hat{q})\th(-\pi\hat{\tau}\frac{\pi}{2\beta},\hat{q}^{\frac{\pi}{\beta}})\th(-\pi\hat{\tau}\frac{z}{\beta},\hat{q}^{\frac{\pi}{\beta}})},
\end{equation}
where $d:=c\exp(-\frac{4i\beta^{2}}{\pi\hat{\tau}})$ is determined by
\begin{equation*}
   d^2=\frac{\th(3\beta,\hat{q})}{\th(\beta,\hat{q})},
\end{equation*}
and $d>0$. The first few terms of $d$ are
\[d=k-k^3(k^2-3)(k^2-1)(k^2+1)\hat{q}^2+O(\hat{q}^4).\]
Using \eqref{eq:Xz_in_qhat}--\eqref{eq:Jz_in_qhat}, we can expand $X(z)$ and $J(z)$ as series in $\mathbb{C}(e^{i\hat{\tau}z})[[\hat{q}]]$ and $\mathbb{C}(e^{\frac{i\pi\hat{\tau} z}{\beta}})[[\hat{q}^2,\hat{q}^{\frac{\pi}{\beta}}]]$, respectively. Writing $\hat{z}=z\hat{\tau}$, we can write $X(z)$ and $J(z)$ as series in $\hat{z}\mathbb{C}[[\hat{z},\hat{q}]]$ and $\frac{1}{\hat{z}}\mathbb{C}[[\hat{z},\hat{q}^2,\hat{q}^{\frac{\pi}{\beta}}]]$, respectively. Writing $u=\frac{\cos(\beta_{0}+2\hat{z})}{\cos(\beta_{0})}-1\in\mathbb{C}[[\hat{z}]]$, these have initial terms
\begin{equation*}
\left\{\begin{array}{rcl}X(z)&=&\displaystyle \frac{ku}{u+3-k^{2}}+k\frac{3-k^{2}}{1+k^{2}}\left(\frac{-4\beta_{1}\sin(2\hat{z})}{(u+3-k^{2})^{2}}+\frac{u(1+k^{2})^2(1-k^{2}+k^{4}+u)}{u+3-k^{2}}\right)\hat{q}^{2}+O(\hat{q}^{4}),\medskip\\
J(z)&=&\displaystyle \frac{\pi}{\beta_{0} k}\frac{\sin(2\beta_{0})}{\sin(\frac{\pi\hat{z}}{\beta_{0}})}+
J_{1}(\hat z)\hat{q}^{2}+\frac{4\pi}{\beta_{0} k}\sin(2\beta_{0})\sin\left(\frac{\pi\hat{z}}{\beta_{0}}\right)\hat{q}^{\frac{\pi}{\beta}}+O(\hat{q}^{4}),\end{array}\right.
\end{equation*}
where
\[J_{1}(\hat z)=\frac{\pi\sin(2\beta_{0})}{\beta_{0} k\sin(\frac{\pi\hat{z}}{\beta_{0}})}\left((1+k^2-k^4)(3-k^2)(1+k^2)
+\frac{\pi\beta_{1}\hat{z}\cos(\frac{\pi\hat{z}}{\beta_{0}})}{\beta_{0}^{2}\sin(\frac{\pi\hat{z}}{\beta_{0}})}+\frac{2\beta_{1}\cos(2\beta_{0})}{\sin(2\beta_{0})}-\frac{\beta_{1}}{\beta_{0}}\right).\]
Taking the inverse of the first series, we can then write $\hat{z}$ as a series in $\mathbb{C}[\hat{q}^2,X(z)]$, which yields $J$ as a series in $\frac{1}{X(z)}\mathbb{C}[[\hat{q}^2,\hat{q}^{\frac{\pi}{\beta}},X(z)]]$. Combining this with \eqref{eq:Jinxz}, this yields $Q(x,0)$ as a series in $\mathbb{C}[[\hat{q}^2,\hat{q}^{\frac{\pi}{\beta}},x]]$. Note, however, that $\beta$ still depends on $\hat{q}$, so to complete our understanding of $Q(x,0)$ we will need to expand $\hat{q}^{\frac{\pi}{\beta}}$ as a series in $\hat{q}$. This is where logarithmic terms will appear, as $\hat{q}^{\frac{\pi}{\beta}}\in\hat{q}^{\frac{\pi}{\beta_{0}}}(1+\hat{q}^2\log(\hat{q})\mathbb{C}[[\hat{q},\hat{q}^2\log(\hat{q})]]).$ In particular, using $\beta=\beta_{0}+\beta_{1}\hat{q}^2+\dots$, we have
\[\hat{q}^{\frac{\pi}{\beta}}=\hat{q}^{\frac{\pi}{\beta_{0}}}\left(1-\frac{2\pi\beta_{1}}{\beta_{0}^{2}}\hat{q}^{2}\log(\hat{q})+\frac{2\pi^{2}\beta_{1}^{2}}{\beta_{0}^{4}}\hat{q}^{4}\log(\hat{q})^{2}+\frac{2\pi(\beta_{1}^{2}-\beta_{0}\beta_{2})}{\beta_{0}^{3}}\hat{q}^{4}\log(\hat{q})+O(\hat{q}^{6-\epsilon})\right).\]
So, finally, $Q(x,0)$ is a series in $\mathbb{C}[[\hat{q}^2,x]]+ \hat{q}^{\frac{\pi}{\beta_{0}}}\mathbb{C}[[\hat{q}^2\log(\hat{q}),\hat{q}^2,\hat{q}^{\frac{\pi}{\beta_{0}}},x]]$. Using the relation between $\hat{q}$ and $t$, we can write $\hat{q}^2$ as a series in $(t-t_{c})\mathbb{C}[[t-t_{c}]]$. Hence $Q(x,0)$ is a series in 
\begin{equation*}
    \mathbb{C}[[t-t_{c},x]]+ (t-t_{c})^{\frac{\pi}{2\beta_{0}}}\mathbb{C}[[(t-t_{c})\log(t-t_{c}),t-t_{c},(t-t_{c})^{\frac{\pi}{2\beta_{0}}},x]].
\end{equation*}
Explicitly, we can write this as:
\begin{equation}
\label{eq:series_expansion_log}
   Q(x,0)=A(x,1-t/t_{c})+\sum_{k,\ell=0}^{\infty}\sum_{m=0}^{\ell}(1-t/t_{c})^{\ell+(k+1)\frac{\pi}{2\beta_{0}}}\log(1-t/t_{c})^{m}P_{k,\ell,m}(x),
\end{equation}
where each $P_{k,\ell,m}(x)\in\mathbb{C}[[x]]$. The $A$ part in \eqref{eq:series_expansion_log} has no effect on the asymptotic expansion, this all comes from the series $P_{k,\ell,m}$. We note that $\beta_{0}\in(\frac{\pi}{4},\frac{\pi}{3})$, so $\frac{\pi}{2\beta_{0}}\in (\frac{3}{2},2)$. So the leading terms in the asymptotic expansion are:
\begin{multline*}
    (1-t/t_{c})^{\frac{\pi}{2\beta_{0}}}P_{0,0,0}(x)+(1-t/t_{c})^{1+\frac{\pi}{2\beta_{0}}}P_{0,1,0}(x)+(1-t/t_{c})^{1+\frac{\pi}{2\beta_{0}}}\log(1-t/t_{c})P_{0,1,1}(x)\\+(1-t/t_{c})^{\frac{\pi}{\beta_{0}}}P_{1,0,0}(x).
\end{multline*}
We can calculate these explicitly, for example the first term is given by
\begin{equation*}
   xP_{0,0,0}(x)=\frac{2\pi k(3+k^2)\sqrt{3+2k^2-k^4}}{\beta_{0}(1-k^2)}\left(\frac{3+k^2}{(3-k^2)^3(1+k^2)^2}\right)^{\frac{\pi}{2\beta_{0}}}\sin\left(\frac{\pi}{\beta_{0}}\hat{z}\right),
\end{equation*}
where $\hat{z}$ and $x$ are related by
\begin{equation*}
   x=k\frac{\sin(\hat{z}-\beta_{0})\sin(\hat{z})}{\sin(\hat{z}-2\beta_{0})\sin(\hat{z}+\beta_{0})} = \frac{ku}{u+3-k^{2}}.
\end{equation*}
It can be checked by a direct computation that the function $\sin\left(\frac{\pi}{\beta_{0}}\hat{z}\right)$ exactly corresponds to the generating function of the positive harmonic function for the model, as computed in \cite{Ra-14}.

The term associated to $P_{0,0,0}(x)$ determines the leading asymptotic behaviour of the coefficients:
\begin{equation}\label{eq:asymptotics_Qx0}[x^{j}][t^n]Q(x,0)\sim C\left([x^{j+1}]\sin\left(\frac{\pi}{\beta_{0}}\hat{z}\right)\right)n^{-1-\frac{\pi}{2\beta_{0}}}t_{c}^{-n}\end{equation}
for fixed $j$ as $n\to\infty$, where $C$ is a constant (only depending on $a$) given by
\begin{equation*}
   C=\frac{2\pi k(3+k^2)\sqrt{3+2k^2-k^4}}{\beta_{0}(1-k^2)\Gamma(-\frac{\pi}{2\beta_{0}})}\left(\frac{3+k^2}{(3-k^2)^3(1+k^2)^2}\right)^{\frac{\pi}{2\beta_{0}}}.
\end{equation*}
 In order to verify that there really is a logarithmic term in this expansion, we also calculate $P_{0,1,1}(x)$ exactly. In fact this only differs from $P_{0,0,0}(x)$ by a constant multiple (dependent on $a$ but not $x$):
\begin{equation*}
   xP_{0,1,1}(x)=\frac{\pi^{2} k^3(k^2+3)(1-2k^2)(3-k^{2})(1+k^2)\left(\frac{3+k^2}{(3-k^2)^3(1+k^2)^2}\right)^{1+\frac{\pi}{2\beta_{0}}}}{2\beta_{0}^{3}}\sin\left(\frac{\pi}{\beta_{0}} \hat{z}\right).
\end{equation*}
The only value of $k$ for which $P_{0,1,1}(x)=0$ is $k=\frac{1}{\sqrt{2}}$, corresponding to $a=\sqrt{2}$, $\beta_{0}=\frac{1}{2}\arccos(-\frac{1}{4})$ and $t_{c}=\frac{\sqrt{2}}{7}$, However we note that there are still logarithmic terms in the asymptotics in this case, for example $P_{0,1,2}(x)\neq 0$.

\subsection{The limit \texorpdfstring{$a\to0$}{when the parameter goes to zero}}
A priori our results only apply for $a>0$, and indeed this is necessary as some functions such as $P_{0,0,0}$ diverge for $a=0$. Nonetheless, we see that the leading asymptotic expression \eqref{eq:asymptotics_Qx0} converges in a way that somewhat corresponds to the $a=0$ case. Note that as $a\to 0$ we also have $k\to 1$, $\beta_{0}\to \frac{\pi}{4}$ and $t_{c}\to\frac{1}{4}$. The limit of the constant $C$ as $a\to 0$ is $\frac{4}{\pi}$. Moreover, we have
\begin{equation*}
   x=\frac{\sin(\hat{z}-\frac{\pi}{4})\sin(\hat{z})}{\sin(\hat{z}-\frac{\pi}{2})\sin(\hat{z}+\frac{\pi}{4})},
\end{equation*}
from which it follows that
\begin{equation*}
   \sin\left(\frac{\pi}{\beta_{0}}\hat{z}\right)=\sin\left(4\hat{z}\right)=\frac{4x}{(1-x)^{2}}.
\end{equation*}
So \eqref{eq:asymptotics_Qx0} would give
\begin{equation*}
   [x^{j}][t^n]Q(x,0)\sim \frac{16}{\pi}(j+1)n^{-3}4^{n}.
\end{equation*}
The only problem with this is that $[x^{j}][t^n]Q(x,0)$ is $0$ when $j$ and $n$ have opposite parity, whereas for terms with the same parity the correct asymptotic formula is
\begin{equation*}
   [x^{j}][t^n]Q(x,0)\sim \frac{32}{\pi}(j+1)n^{-3}4^{n}.
\end{equation*}
In other words, this correctly yields the behaviour of $Q(x,0)$ around $t=t_{c}=\frac{1}{4}$, however there is a second critical point, $-\frac{1}{4}$ on the radius of convergence.

\section{Polyharmonicity of coefficients}
\label{sec:polyharmonicity}

\noindent Due to \eqref{eq:series_expansion_log} we already have a fair amount of information about the coefficients of $Q(x,0)$ at the critical point. One can now use this in order to describe the asymptotic behaviour of the (weighted) number of paths $q((0,0),(i,j);n):=q(i,j;n)$ in the quadrant from the origin to $(i,j)$ with $n$ steps, see \eqref{eq:def_series}. In particular, we will see in Lemma~\ref{lem:expansion_q(s,t)} that the dependence on the number of steps $n$ is given in terms of a mix of powers of $n$ and logarithms. In a similar fashion as in~\cite{Ne-23}, one can then show that the dependence on the endpoint $(i,j)$ is given in terms of so-called discrete polyharmonic functions (see \cite{Ne-22,Ne-23,ArNaCr-83,Ra-14,ChFuRa-20}).

Given a step-set $\mathcal{S}$ with corresponding weights $\left(\omega_s\right)_{s\in\mathcal{S}}$, we define a discrete Laplacian operator $\triangle$ acting on functions $v : \mathbb Z^2\to \mathbb C$  as follows:
\begin{equation*}
    \triangle v ( A ) = \sum_{s\in\mathcal S} \omega_s v(A+s) - t v(A).
\end{equation*}
We say that a function $v$ is $t$-harmonic (resp.\ $t$-polyharmonic of order $k$, for some positive integer $k$) if for all points $A$ in the quarter plane, $\triangle v(x)=0$ (resp.\ $\triangle^k v(x)=0$).
Furthermore, in order to keep the notation compact in the following, let 
\begin{equation*}
    \rho:=\frac{\pi}{2\beta_0},
\end{equation*}
where $\beta_0$ is defined as in \eqref{eq:relation_k_beta}. As mentioned in the previous section, $\rho$ varies (continuously) in $\left(\frac{3}{2},2\right)$ as $a$ varies in $(0,\infty)$.

Our main objective in this section is to show the following result:
\begin{thm}
\label{thm:coefficients_PHF}
If $\rho\notin\Q$, then for any $p>0$ (not necessarily integer), we have 
\begin{equation*} \label{eq:coefficients_PHF}   q(i,j;n)=t_c^{-n}\sum_{\substack{k\geq 1,\ell\geq m\geq 0\\k\rho+\ell+1<p}}v_{k,\ell,m}(i,j)\frac{(\log{n})^m}{n^{k\rho+\ell+1}}+\mathcal{O}\left(t_c^{-n} \left(\frac{\log{n}}{n}\right)^p\right),
\end{equation*}
where the $v_{k,\ell,m}$ are discrete $t_c$-polyharmonic functions of order $\ell-m+1$.
If $\rho=\frac{u}{v}\in\Q$ with $u$ and $v$ coprime, then the same holds with the additional condition that the summation index $k$ be at most $v$. 
\end{thm}

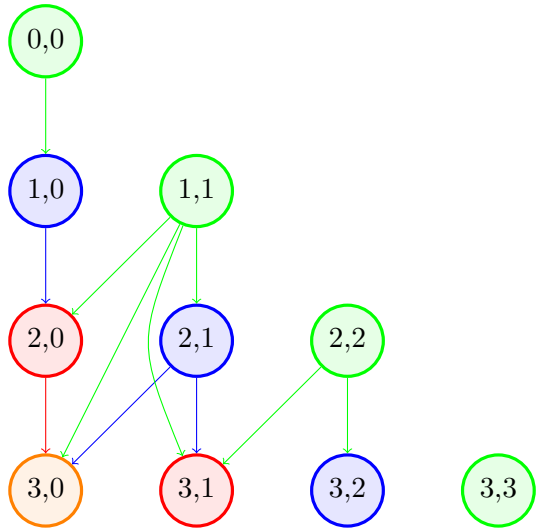
\begin{figure}[ht]
\centering
\begin{tikzpicture}[
roundnode/.style={circle, draw=red!100, fill=red!10, very thick, minimum size=7mm},
biharmnode/.style={circle, draw=blue!100, fill=blue!10, very thick, minimum size=7mm},
harmnode/.style={circle,draw=green!100, fill=green!10, very thick, minimum size=7mm},
4harmnode/.style={circle,draw=orange!100, fill=orange!10, very thick, minimum size=7mm}
]
\node[harmnode]      (00)                              {0,0};
\node[biharmnode]        (10)       [below=of 00] {1,0};
\node[roundnode]      (20)       [below=of 10] {2,0};
\node[harmnode]        (11)       [right=of 10] {1,1};
\node[biharmnode] (21)           [right=of 20] {2,1};
\node[harmnode] (22)           [right=of 21] {2,2};
\node[4harmnode] (30) [below=of 20] {3,0};
\node[roundnode] (31) [below=of 21] {3,1};
\node[biharmnode] (32) [below=of 22] {3,2};
\node[harmnode] (33) [right=of 32] {3,3};

\path[->, green] 
(00) edge (10)
(11) edge (30)
(11) edge (20)
(11) edge (21)
(22) edge (31)
(22) edge (32);

\path[->,blue]
(10) edge (20)
(21) edge (30)
(21) edge (31);

\path [->,red]
(20) edge (30);

\draw[->,green]
(11) .. controls (1.2,-4) .. (31);
\end{tikzpicture}
\caption{ The arrows indicate a change in the coefficients of $f_{\ell,m}:=\frac{(\log{n})^m}{n^\ell}$ when substituting \eqref{eq:PHF_n+1_to_n} into \eqref{eq:lemma_PHF_condition2}. We see that $f_{0,0}$, $f_{1,1}$, $f_{2,2}$ and $f_{3,3}$ (the green nodes) are not affected as we can only ever go down and to the left as indicated (see Lemma~\ref{lem:n+1_to_n}); thus their coefficients $v_{0,0}$, $v_{1,1}$, etc., are harmonic. Drawing the same diagram for $\triangle q(x;n)$, these nodes therefore disappear, and the unaffected coefficients are then those of the blue nodes. Thus, $v_{1,0}, v_{2,1}$, \dots, are biharmonic, and so on.} 
\label{fig:PHF_coefficients}
\end{figure}

\noindent See Figure~\ref{fig:PHF_coefficients} for an illustration of Theorem~\ref{thm:coefficients_PHF} for fixed $k$, showing in particular the inter-dependency of the polyharmonic functions $v_{k,\ell,m}$. For irrational $\rho$ we will have infinitely many such diagrams; whereas for rational $\rho$ there will be only finitely many.

The rest of Section~\ref{sec:polyharmonicity} is devoted to the proof of Theorem~\ref{thm:coefficients_PHF}. In the following we will assume that $\rho\notin\Q$; otherwise we only need to bound the $k$ in the summation indices by its denominator (as in Theorem~\ref{thm:coefficients_PHF}). Using a standard transfer between the local behaviour of the generating function around the singularity and the asymptotics of the coefficients as in~\cite[VI.2]{FlSe-09}, we know due to \eqref{eq:series_expansion_log} that we have an asymptotic expansion of $q(i,0;n)$ and $q(0,j;n)$ (which are identical due to the symmetry of the model; this is however not necessary for the following) of the form 
\begin{equation}\label{eq:expansion_boundary}
q(i,0;n)=t_c^{-n}\sum_{\substack{k\geq 1,\ell\geq m\geq 0\\k\rho+\ell<p}}v_{k,\ell,m}(i,0)\frac{(\log{n})^m}{n^{k\rho +\ell+1}}+\mathcal{O}\left(t_c^{-n}\left(\frac{\log{n}}{n}\right)^p\right),
\end{equation}
for some coefficients denoted by $v_{k,\ell,m}(i,0)$. The structure of the proof of Theorem~\ref{thm:coefficients_PHF} is as follows: first, we want to show that a similar expansion holds not only for $q(i,0;n)$, but also for $q(i,j;n)$, which will be done in Lemma~\ref{lem:expansion_q(s,t)}. Then we will make use of this in order to show that the resulting coefficients $v_{k,\ell,m}(i,j)$ are discrete polyharmonic functions, see Lemma~\ref{lemma:PHF}.

Before we start, we will state a simple lemma, which will turn out to be useful:
\begin{lem}
\label{lem:n+1_to_n}
For any integer $p>\ell$, we have
\begin{equation}
\label{eq:PHF_n+1_to_n}\frac{(\log(n+1))^m}{(n+1)^\ell}=\frac{(\log{n})^m}{n^\ell}+\sum_{i=1}^{p-\ell}\sum_{j=\max(m-i,0)}^m c_{\ell,m,i,j}\frac{(\log{n})^j}{n^{\ell+i}}+\mathcal{O}\left(\frac{(\log{n})^m}{n^{p+1}}\right),
\end{equation}
with the $c_{\ell,m,i,j}$ constants (in particular, if $m=0$, then we will have no logarithmic parts).
\end{lem}
\begin{proof}
The result follows immediately from writing $\log(n+1)=\log(n)+\log(1+\frac{1}{n})$ and expanding as a series in $\frac{1}{n}$.
\end{proof}

\noindent Note that in particular all terms inside the sum on the right-hand side have powers of the logarithm not exceeding $m$, and powers of $n$ strictly smaller than $-\ell$ (this is the reason why in Figure~\ref{fig:PHF_coefficients} all arrows can only go downwards, and possibly to the left).

We can now extend the asymptotic expansion of $q(i,0;n)$ as in \eqref{eq:expansion_boundary} to one of $q(i,j;n)$. 
\begin{lem}
\label{lem:expansion_q(s,t)}
    For any $(i,j)\in\Z^2$ and any $p>0$, we have 
    \begin{equation}\label{eq:expansion_q(s,t;n)}
        q(i,j;n)=t_c^{-n}\sum_{\substack{k\geq 1, \ell\geq m\geq 0\\ k\rho+\ell+1<p}}v_{k,\ell,m}(i,j)\frac{(\log n)^m}{n^{k\rho+\ell+1}}+\mathcal{O}\left(t_c^{-n}\left(\frac{\log{n}}{n}\right)^p\right).
    \end{equation}
\end{lem}
\begin{proof}
    By induction on $r:=\min(i,j)$. For $r<0$, the statement is trivial because we have $q(i,j;n)=0$ (we will use repeatedly this convention throughout the proof). For $r=0$, the statement is precisely \eqref{eq:expansion_boundary}. So let us suppose that we already know that \eqref{eq:expansion_q(s,t;n)} holds up to a given $r$, and consider a point $(i,r+1)$. We can then write 
    \begin{multline*}
        q(i,r;n+1)=\\q(i,r+1;n)+q(i,r-1;n)+q(i-1,r;n)+q(i+1,r;n)+a q(i-1,r-1;n).
    \end{multline*} 
    By induction hypothesis and Lemma~\ref{lem:n+1_to_n} applied to $q(i,r;n+1)$, the statement follows.
\end{proof}

\noindent In the following Lemmas~\ref{lemma:PHF_preparatory} and~\ref{lemma:PHF}, we will formalize the argumentation given in Figure~\ref{fig:PHF_coefficients}.
By \cite{Ne-23} we know that, ordering the triples $(k,\ell,m)$ in \eqref{eq:expansion_q(s,t;n)} such that the weight functions $f_{k,\ell,m}(n):=\frac{(\log{n})^m}{n^{k+\ell\rho+1}}$ are decreasing, then the corresponding coefficient functions $v_{k,\ell,m}$ are polyharmonic functions of increasing order; i.e.~$v_{1,0,0}$ is harmonic, $v_{1,1,0}$ is biharmonic, $v_{1,1,1}$ is triharmonic, and so on. It turns out, however, that the fact that we know the $f_{k,\ell,m}$ explicitly allows us to greatly improve upon this statement.

\begin{lem}
    \label{lemma:PHF_preparatory}
    Suppose we have $p>0$ and a sequence of reals $a_{k,\ell,m}$ such that \begin{equation*}
        \sum_{\substack{k\geq 1,m\geq \ell\geq 0\\k\rho+\ell< p}}\frac{a_{k,\ell,m}}{n^{k\rho+\ell+1}}(\log{n})^m=\mathcal{O}\left({n^{-p}}\right).
    \end{equation*}
    Then, $a_{k,\ell,m}=0$ for all triples $(k,\ell,m)$.
\end{lem}
\begin{proof}
    After multiplying with $n^{\rho+1}$, taking the limit $n\to\infty$ we can see immediately that $a_{1,0,0}=0$. Proceeding by multiplying with increasing powers of $n$ and $\log n$ and using that by assumption in each case there is only one non-zero coefficient (note that we make use of the fact that $\rho\notin\Q$ here), we can inductively show that $a_{k,\ell,m}=0$ for all triples $(k,\ell,m)$. 
\end{proof}

\noindent The idea behind the following Lemma~\ref{lemma:PHF} is similar as in \cite[Lem.~6]{Ne-23}. One writes $q(x;n+1)$ recursively as a sum over the step-set, and utilizes Lemma~\ref{lem:n+1_to_n} in order to compare the coefficients. For large $n$, many of these coefficients will disappear, leaving us essentially with the (poly-)harmonicity of the $v_{k,\ell,m}$. 

\begin{lem}
\label{lemma:PHF}
    Suppose we have a combinatorial quantity $q(x;n)$ such that we have 
    \begin{equation}        \label{eq:lemma_PHF_condition1}
        q(x;n+1)=\sum_{s\in\mathcal{S}}\omega_s q(x-s;n),
        \end{equation}
and for each $p>0$ we have
        \begin{equation}        \label{eq:lemma_PHF_condition2}      q(x;n)=t_c^{-n}\sum_{\substack{k\geq 1,\ell\geq m\geq 0\\k\rho+\ell+1< p}}v_{k,\ell,m}(x)\frac{(\log{n})^m}{n^{k\rho+\ell+1} }+\mathcal{O}\left(t_c^{-n}\left(\frac{\log{n}}{n}\right)^p\right).
        \end{equation}
    Then $v_{k,\ell,m}$ is $t_c$-polyharmonic of order $\ell-m+1$. 
\end{lem}
\begin{proof}
To shorten notation, define the sets \begin{align*}
    \mathcal{U}:=&\{(k,\ell,m)\in\Z^3: k\geq 1, \ell\geq m\geq 0, k\rho + \ell + 1 <p\},\\ \mathcal{V}_{k,\ell,m}:=&\{(i,j)\in\Z^2: 1\leq i\leq m, k\rho + \ell + i+1 < p,m-i\leq j \leq m\}.
\end{align*}
First, we notice that by \eqref{eq:PHF_n+1_to_n} and \eqref{eq:lemma_PHF_condition2}, we have \begin{multline*}
     q(x;n+1)=\\ t_c^{-n-1}\sum_{(k,\ell,m)\in\mathcal{U}} v_{k,\ell,m}(x)\left(\frac{(\log{n})^m}{n^{k\rho+\ell+1}}+\sum_{(i,j)\in\mathcal{V}_{k,\ell,m}}c_{k,\ell,m,i,j}\frac{(\log{n})^{j}}{n^{k\rho+\ell+i+1}}\right)+\mathcal{O}\left(t_c^{-n}\left(\frac{\log{n}}{n}\right)^p\right)
 \end{multline*}
 for some constants $c_{k,\ell,m,i,j}$. Utilizing \eqref{eq:lemma_PHF_condition1}, we now obtain \begin{multline}
  \label{eq:lemma_PHF_eq2}
     t_c^{-n}\sum_{s\in\mathcal{S}}\sum_{(k,\ell,m)\in\mathcal{U}} \omega_s v_{k,\ell,m}(x-s)\frac{(\log{n})^m}{n^{k\rho+\ell+1}}=\\t_c^{-n-1}\sum_{(k,\ell,m)\in\mathcal{U}} v_{k,\ell,m}(x)\left(\frac{(\log{n})^m}{n^{k\rho+\ell+1}}+\sum_{(i,j)\in\mathcal{V}_{k,\ell,m}}c_{k,\ell,m,i,j}\frac{(\log{n})^{j}}{n^{k\rho+\ell+i+1}}\right)+\mathcal{O}\left(t_c^{-n}\left(\frac{\log{n}}{n}\right)^p\right).
 \end{multline}
 Now let us partially order the triples $(k,\ell,m)$ giving them the index $\ell-m$ (which means sorting them by their diagonal in Figure~\ref{fig:PHF_coefficients}).
 We proceed inductively by the index of the triples $(k,\ell,m)$.
 
 For index $0$, one can check immediately (using an argument as shown in Figure~\ref{fig:PHF_coefficients}; formally utilizing \eqref{eq:PHF_n+1_to_n}) that the coefficient of $\frac{(\log{n})^m}{n^{k\rho+\ell+1}}$ in the right-hand side of \eqref{eq:lemma_PHF_eq2} is precisely $v_{k,\ell,m}(x)$. From Lemma~\ref{lemma:PHF_preparatory} it follows immediately that the corresponding coefficients $v_{k,\ell,m}$ are $t_c$-harmonic.
 
 Now suppose the statement is already shown for all triples of order $r$, and consider those of order $r+1$. We utilize the same arguments as before on $\triangle^r q(x;n)$. The equivalent of \eqref{eq:lemma_PHF_eq2} now has the form 
 \begin{multline*}
     t_c^{-n}\sum_{s\in\mathcal{S}}\sum_{(k,\ell,m)\in\mathcal{U}} \omega_s \triangle^rv_{k,\ell,m}(x-s)\frac{(\log{n})^m}{n^{k\rho+\ell+1}}=\\t_c^{-n-1}\sum_{(k,\ell,m)\in\mathcal{U}} \triangle^rv_{k,\ell,m}(x)\left(\frac{(\log{n})^m}{n^{k\rho+\ell+1}}+\sum_{(i,j)\in\mathcal{V}_{k,\ell,m}}c_{k,\ell,m,i,j}\frac{(\log{n})^{j}}{n^{k\rho+\ell+i+1}}\right)+\mathcal{O}\left(t_c^{-n}\left(\frac{\log{n}}{n}\right)^p\right),
 \end{multline*}
 where we let the sum run over the triples with index at least $r+1$, since by induction hypothesis $\triangle^rv_{k,\ell,m}(x)=0$ for all $(k,\ell,m)$ with index at most $r$. But from here it is again easily seen that the coefficient of $\frac{(\log{n})^m}{n^{k\rho+\ell+1}}$ is nothing else than $\triangle^rv_{k,\ell,m}$ for the triples of index $r$; thus $\triangle^rv_{k,\ell,m}$ is harmonic and the proof is complete. 
\end{proof}

Finally, Theorem~\ref{thm:coefficients_PHF} follows from Lemmas~\ref{lem:expansion_q(s,t)} and~\ref{lemma:PHF}. Note that for the proof we did not make use of the fact that we are in dimension $2$; the only part where we used any properties of our model in particular was in the proof of Lemma~\ref{lem:expansion_q(s,t)}.

\appendix 
\section{Proof of Lemmas~\ref{lemma:q_values} and \ref{lem:real_sing}}\label{app:lemma:q_values}
\noindent We start by showing Lemma~\ref{lemma:q_values}.

\begin{proof}
    From \cite{BoBo-87} we know that we can write \begin{equation*}
        q=\exp\left(-\pi K\left(\sqrt{1-k^2}\right)/K\left(k\right)\right),
    \end{equation*}
    where $k$ is the elliptic modulus, with $0<k<1$, and $K(k)$ is the complete elliptic integral \begin{equation*}
        K(k)=\int_0^{\pi/2}\frac{\mathrm{d}\theta}{\sqrt{1-k^2\sin^2{\theta}}}.
    \end{equation*}
    Using the explicit formula \cite[(7.26)]{KuRa-12} for $k$, we see that $k\in(0,1)$ for $t\in(0,t_c)$, so the first point follows. Making use of the fact that zeros of a polynomial are generically continuous in its coefficients, we see that $q(t)$ is continuous for $t\in(0,t_c)$.
     
    It is shown in \cite[7.4]{KuRa-12} that if $t\to 0$, then $k\to 1$, which implies that $\lim_{t\to 0}q(t)=0$. By the same argument we can see that, using the notation as in \cite{FaIaMa-17,KuRa-12}, if as $t\to t_c$ we have $x_2\to x_3$, then forcibly $\lim_{t\to t_c}q(t)=1$. But this can be seen with a straightforward adaptation of \cite[Sec.~2.3]{FaIaMa-17} (note that for our model the discriminant vanishes for no $t$ at either $0$ or $\infty$).
 \end{proof}

\noindent\textbf{Remark:} While not needed here, it seems plausible that $q(t)$ is in fact increasing for $t\in(0,t_c)$. A proof of this would likely come down to the study of the zeros $x_1$ to $x_4$ of the discriminant (see e.g.~\cite{KuRa-12}) and get rather technical.

\medskip

\noindent We will now proceed to prove Lemma~\ref{lem:real_sing}. To do so, we first show the preparatory Lemma~\ref{lem:lsingularity}.

\begin{lem}\label{lem:lsingularity} Let $\mathcal S$ be a weighted step-set, let $Q(x,y)$ be the generating function \eqref{eq:def_series} for walks in the quadrant using this step-set and let $r$ be the radius of convergence of $Q(0,0)$. If $\ell\in\mathbb{N}$ satisfies $[t^\ell] Q(0,0)\neq 0$, then each generating function $[x^a][y^b]Q(x,y)$ has no singularities $t_{c}$ for $\vert t_{c}\vert \leq r$ except possibly at points $re^{\frac{2\pi i j}{\ell}}$ for $j=0,1,\ldots,\ell-1$.
\end{lem}
\begin{proof}
We start by writing
\begin{equation*}
   Q(x,y;t)=Q_{0}(x,y;t^{\ell})+tQ_{1}(x,y;t^{\ell})+\cdots+t^{\ell-1}Q_{\ell-1}(x,y;t^{\ell}),
\end{equation*}
so $t^{j}Q_{j}(x,y;t^{\ell})$ counts walks whose length is $j$ more than a multiple of $\ell$. The radius of convergence of $[x^a y^b]t^{j}Q_{j}(x,y;t^{\ell})$ is no less than $r$ as its coefficients are bounded above by the coefficients of $[x^a y^b]Q(x,y;t)$. Hence each $[x^a y^b]Q_{j}(x,y;t)$ has no singularities in $\vert t\vert \leq r^{\ell}$. We will now prove that each $[x^a y^b]Q_{j}(x,y;t)$ has no singularities in $\vert t\vert =r^{\ell}$ except possibly at $t=r^{\ell}$. Note that $[x^a y^b]Q_{j}(x,y;t)$ counts weighted quadrant walks from $(0,0)$ to $(a,b)$ which can be cut into successive walks $w_1,w_2,w_3,\ldots,w_m$ of length $\ell$ followed by a walk $v$ of length $j$, where $t$ counts the number of walks of length $\ell$. 
Let $\Omega$ be a path of length $\ell$ from $(0,0)$ to $(0,0)$ in the quadrant and let $w_{\Omega}>0$ be the weight of $\Omega$. Now let $\tilde{Q}_{j}(x,y;t)$ be the generating function for walks counted by $Q_{j}(x,y;t)$, where none of the subpaths $w_{k}$ is equal to $\Omega$. Note then that any walk counted by $Q_{j}(x,y;t)$ can be uniquely constructed by taking a walk counted by $\tilde{Q}_{j}(x,y;t)$ and inserting any number of copies of $\Omega$ before each $w_{k}$ and $v$. Hence
\begin{equation*}
   Q_{j}(x,y;t)=\frac{1}{1-w_{\Omega}t}\tilde{Q}_{j}\left(x,y,\frac{t}{1-w_{\Omega}t}\right).
\end{equation*}
Now, the series $F(t):=\tilde{Q}_{j}(x,y;t)$ has non-negative coefficients, so its radius of convergence $t_{F}>0$ is a singularity of $F(t)$. Moreover, $F(t)$ satisfies
\begin{equation*}
   [x^a][y^b]Q_{j}(x,y;t)=\frac{1}{1-w_{\Omega}t}F\left(\frac{t}{1-w_{\Omega}t}\right).
\end{equation*}
Hence $[x^a][y^b]Q_{j}(x,y;t)$ has a corresponding singularity at $t_{Q}=\frac{t_{F}}{1+w_{\Omega}t_{F}}>0$, so we must have $t_{Q}\geq r^{\ell}$. If $[x^a][y^b]Q_{j}(x,y;t)$ has another singularity $t_{0}$ satisfying $\vert t_0\vert =r^{\ell}$, then $\frac{t_{0}}{1-w_{\Omega}t_{0}}$ is a singularity of $F$, so
\begin{equation*}
   \frac{r^{\ell}}{\vert 1-w_{\Omega}t_{0}\vert }=\frac{\vert t_{0}\vert }{\vert 1-w_{\Omega}t_{0}\vert }\geq t_{F}=\frac{t_{Q}}{1-w_{\Omega}t_{Q}}\geq \frac{r^{\ell}}{1-w_{\Omega}t_{Q}},
\end{equation*}
hence
\begin{equation*}
   \vert 1-w_{\Omega}t_{0}\vert \leq 1-w_{\Omega}t_{Q}\leq 1-w_{\Omega}r^{\ell}=1-\vert w_{\Omega}t_{0}\vert .
\end{equation*}
By the triangle inequality, this is only possible if $w_{\Omega}t_{0}>0$, i.e.~$t_0=r^{\ell}$. Hence $[x^a][y^b]Q_{j}(x,y;t)$ has no other singularities $t_0$ satisfying $\vert t_0\vert =r^{\ell}$. Therefore, the series $[x^a][y^b]t^{j}Q_{j}(x,y;t^{\ell})$ has no singularities on the radius of convergence $r$ except possibly at points $re^{\frac{2\pi i j}{\ell}}$ for $j=0,1,\ldots,\ell-1$. Since this is true for all $j$, it follows that the same statement holds for $[x^a][y^b]Q(x,y;t)$.
\end{proof}

\noindent We can now proceed with the proof of Lemma~\ref{lem:real_sing}.
\begin{proof}
First, since $Q_{a,b}(t):=[x^a][y^b]Q(x,y)$ is non-constant, it must have the same radius of convergence $r$ as $Q(0,0)$. Moreover, since $Q_{a,b}(t)$ has only non-negative coefficients, $r$ must be a singularity. Since $k$ is a period of the model, the powers of $t$ appearing in the generating function $Q_{a,b}(t)$ must all have the same residue $u$ modulo $k$. This means that for each integer $j$, we can write
\begin{equation*}
   Q_{a,b}(e^{\frac{2\pi i j}{k}}t)=e^{\frac{2\pi i ju}{k}}Q_{a,b}(t),
\end{equation*}
so $re^{\frac{2\pi i j}{k}}$ is a singularity of $Q_{a,b}(t)$, as claimed. Now suppose for the sake of contradiction that there is some other singularity $r\kappa$ on the radius of convergence, with $\vert \kappa\vert =1$ but $\kappa^k\neq 1$. Let $\eta>0$ be minimal such that $\kappa^\eta=1$ (with $\eta:=\infty$ and $\eta\mathbb{Z}:=\{0\}$ if $\kappa$ is not a root of unity). Then $k\notin\eta\mathbb{Z}$ since $\kappa^k\neq 1$. It follows from the maximality of $k$ that there is some $\ell$ satisfying $[t^\ell]Q(0,0)\neq 0$ and $\ell\notin\eta\mathbb{Z}$. From Lemma~\ref{lem:lsingularity}, the singularity $r\kappa$ must satisfy $\kappa^{\ell}=1$, but this is a contradiction as $\ell\notin\eta\mathbb{Z}$.
\end{proof}

\end{document}